\documentclass[11pt,reqno,a4paper,oneside]{amsart}

\usepackage{amssymb}
\usepackage{amscd}
\usepackage{amsfonts}
\usepackage{mathrsfs}
\usepackage{setspace}
\usepackage{version}
\usepackage{mathtools}
\usepackage{dsfont}

\usepackage{a4,amsmath,amssymb,amscd,verbatim,bbm,graphicx,enumerate,tikz}

\definecolor{crimson}{rgb}{0.86, 0.08, 0.24}
\definecolor{bleudefrance}{rgb}{0.14, 0.4, 0.9}
\usepackage[pdftex,colorlinks,linkcolor=crimson,citecolor=bleudefrance,filecolor=black]{hyperref}

\usepackage{float}
\usepackage{tikz-cd}

\usepackage{tikz}
\usetikzlibrary{calc}
\usetikzlibrary{shadings,intersections}
\usetikzlibrary{decorations.text}
\usepackage{pgfplots,wrapfig}
\usepackage{lipsum}

\usepackage{xcolor}

\theoremstyle{plain}
\numberwithin{equation}{section} 
\newtheorem{theorem}{Theorem}[section]
\newtheorem{proposition}[theorem]{Proposition}
\newtheorem{lemma}[theorem]{Lemma}

\newtheorem{corollary}[theorem]{Corollary}
\newtheorem{definition}[theorem]{Definition}
\newtheorem{claim}[theorem]{Claim}
\newtheorem{conjecture}[theorem]{Conjecture}

\theoremstyle{remark}
\newtheorem{remark}[theorem]{Remark}

\newtheorem{example}[theorem]{Example}
\usepackage{soul}

\renewcommand{\leq}{\leqslant}
\renewcommand{\geq}{\geqslant}
\newsavebox{\proofbox}
\savebox{\proofbox}{\begin{picture}(7,7)  \put(0,0){\framebox(7,7){}}\end{picture}}

\newcommand\Z{\mathbb{Z}}
\newcommand\R{\mathbb{R}}

\newcommand\N{\mathbb{N}}

\newcommand\SL{\operatorname{SL}}
\newcommand\SO{\operatorname{SO}}
\newcommand\PSL{\operatorname{PSL}}

\newcommand\im{\operatorname{im}}
\newcommand\Co{\operatorname{Co}}

\newcommand\GL{\operatorname{GL}}

\newcommand\PGL{\operatorname{PGL}}

\newcommand\erg{\operatorname{erg}}
\newcommand\Supp{\operatorname{Supp}}

\newcommand\Hom{\operatorname{Hom}}
\newcommand\Endo{\operatorname{End}}

\newcommand\diag{\operatorname{diag}}
\renewcommand\P{\mathbb{P}}

\newcommand\calG{\mathcal{G}}
\newcommand\calC{\mathcal{C}}

\newcommand\eps{\varepsilon}
\newcommand\id{\operatorname{id}}
\newcommand{\bbS}{\mathbb{S}^1}
\newcommand{\sg}{\mathrm{sg}}
\newcommand{\z}{w}

\newcommand\acts{\operatorname{\curvearrowright}}

\newcommand{\red}[1]{{\color{red}#1}}

\usepackage{accents}
\newlength{\dhatheight}

\begin{document}

\title[Stationary measures on homogeneous bundles over flag varieties]{Stationary measures for $\SL_2(\R)$-actions on homogeneous bundles over flag varieties}

\author{Alexander Gorodnik}
\address{Institut f\"{u}r Mathematik, Universit\"{a}t Z\"{u}rich,  8057 Z\"{u}rich, Switzerland}
\email{alexander.gorodnik@math.uzh.ch, lijialun36@gmail.com, sertcagri@gmail.com}
\thanks{}

\author{Jialun Li}
\email{}
\thanks{}

\author{Cagri Sert}
\email{}
\thanks{
A.G. and J.L. were supported by the SNF grant 200021--182089; C.S. was supported by SNF Ambizione grant 193481}

\begin{abstract}
Let $G$ be a real semisimple Lie group with finite centre and without compact factors, $Q<G$ a parabolic subgroup and $X$ a homogeneous space of $G$ admitting an equivariant projection on the flag variety $G/Q$ with fibres given by copies of lattice quotients of a semisimple factor of $Q$.  
Given a probability measure $\mu$, Zariski-dense in a copy of $H=\SL_2(\R)$ in $G$, we give a description of $\mu$-stationary probability measures on $X$ and prove corresponding equidistribution results. Contrary to the results of Benoist--Quint corresponding to the case $G=Q$, the type of stationary measures that $\mu$ admits depends strongly on the position of $H$ relative to $Q$. We describe possible cases and treat all but one of them, among others using ideas from the works of Eskin--Mirzakhani and Eskin--Lindenstrauss.
\end{abstract}

\maketitle


\section{Introduction}

Let $G$ be a real Lie group and $R<G$ a closed subgroup. The actions of subgroups of $G$ on the homogeneous space $X=G/R$ constitute a natural class of dynamical systems whose (topological, statistical etc.) properties are of key relevance to various problems in mathematics. Accordingly, the study of such dynamical systems has a rich history; it has prompted the introduction of various new techniques and contains major results. The nature of these systems varies according to the acting subgroup and the group $R$ to be factored out, ranging over classes such as partially hyperbolic, parabolic, proximal dynamics, etc. 

One type of ambient homogeneous space $X$ is obtained by considering quotients by discrete subgroups $R=\Lambda<G$. For the actions of connected subgroups of $G$ generated by unipotents (e.g.~ semisimple subgroups) on these quotients $X$, settling conjectures of Raghunathan and Dani, definitive results were obtained by Ratner \cite{ratner.measure.class,ratner.topological}. Her results can be considered as a vast generalization of classical results on vector flows on the tori $\mathbb{T}^d$ and have far-reaching consequences. A key step/result in Ratner's works --- corresponding to Dani's conjecture --- is the classification of measures invariant under unipotent flows. To obtain this result, Ratner introduced an important technique, the polynomial drift argument. 

Setting aside the actions of commuting diagonal flows, a next major step for actions on quotients $X$ by lattices $\Lambda<G$ is reached by the seminal work of Benoist--Quint \cite{BQ1,bq.non-escape,BQ2,BQ3}. Their work involved describing dynamics of actions by  subgroups $\Gamma$ whose algebraic (Zariski) closure has semisimplicity but the subgroups themselves can be genuinely irregular, e.g.~ discrete. 
Focusing on stationary measures of random walks on quotients $X$, they developed the exponential drift argument used to obtain a description of all stationary measures. The drift argument of Benoist--Quint requires a \textit{precise control of random matrix products} (e.g.~ local limit theorem), a feature not readily available without a semisimplicity assumption. Benoist--Quint's ideas were then remarkably modified in a non-homogeneous setting by Eskin--Mirzakhani \cite{eskin-mirzakhani} who managed to set up a much more flexible argument bypassing for example the need for a local limit theorem. This development enabled further extensions of Benoist--Quint's results in several directions by Eskin--Lindenstrauss \cite{eskin-lindenstraus.short,eskin-lindenstrauss.long}. Some of our arguments in this article (e.g.~ the 
six points drift argument) draws on the ideas of the latter works \cite{eskin-mirzakhani,eskin-lindenstraus.short} and in fact can be seen as a slightly modified and simpler version of them.

Continuing to expound elements of our setting, a second type of homogeneous space is obtained by considering quotients by parabolic subgroups $R=Q<G$, giving rise to flag varieties $X=G/Q$. The dynamics on these quotients are quite different from those on quotients by discrete subgroups; in particular, when the acting group has semisimple (non-compact) Zariski-closure, the action is proximal (if the group is split) and the space supports no invariant measures. Starting with the pioneering works of Furstenberg \cite{furstenberg.poisson, furstenberg.nc, furstenberg.boundary.theory} on random matrix products and boundary theory, a thorough qualitative description of dynamics is established by Guivarc'h--Raugi \cite{guivarch-raugi.isom.ext} and Benoist--Quint \cite{BQ.compositio}.

The homogeneous space $X=G/R$ considered in this article is a combination of the two types of classical homogeneous spaces discussed above; it has the structure of a fibre bundle over the flag variety $G/Q$ with fibres given by a homogeneous space $S/\Lambda$, where $S$ is a semisimple group and $\Lambda$ is a discrete subgroup. To illustrate and motivate this structure, recall that a standard example for the first kind of spaces (obtained as quotients by discrete subgroups) is provided by the space $X_{d,d}$ of rank-$d$ lattices in $\R^d$ up to homothety; which can be identified with $\PGL_d(\R)/\PGL_d(\Z)$. Now when one considers more generally the space $X_{k,d}$ of rank-$k$ lattices in $\R^d$ up to homothety, if $k \neq d$, then $X_{k,d}$ has a natural structure of a bundle over the space of $k$-Grassmannians in $\R^d$ (which is a standard example of a space realized as a quotient by a parabolic subgroup) with fibres given by copies of $\PGL_k(\R)/\PGL_k(\Z)$.

The study of dynamics on these quotients is initiated by the work of Sargent--Shapira \cite{sargent-shapira}. Generalizing arguments of Benoist--Quint \cite{BQ1,bq.non-escape}, they managed to describe the dynamics on the space $X_{2,3}$ when the acting probability measure is Zariski-dense in $\SL_3(\R)$ or in an irreducible copy of $\PGL_2(\R)$. Remarkably, they discovered\footnote{interestingly, with a computer experiment} a somewhat unexpected phenomenon (a $\Gamma_\mu$-invariant section, see \cite{sargent-shapira}) in the latter case, a precise understanding of which was an initial motivation for our work. The goal of the current article is more generally to obtain measure classification and equidistribution results in all possible situations\footnote{We manage this except in one case, Case 2.3.b, see Figure \ref{figure.cases} and the discussion below.} when the acting probability measure is Zariski-dense in a copy of $\SL_2(\R)$ or $\PGL_2(\R)$ and when the ambient space $G/R$ has minimal assumptions.

Among others, the results of our work show that in contrast to the type of results obtained by Benoist--Quint, a variety of various dynamical situations are possible even for the actions of groups such as $\SL_2(\R)$ (see Figure \ref{figure.cases}). Moreover, for some of these cases the exponential drift argument of Benoist--Quint is not applicable as such and indeed we develop a different drift argument inspired from those of Eskin--Mirzakhani \cite{eskin-mirzakhani} and Eskin--Lindenstrauss \cite{eskin-lindenstraus.short}. Alternatively, we also demonstrate that a precise control on random matrix products (such as a uniform renewal theorem) can also be used to obtain measure classification. Even though the actions on fibres and base individually are well-understood in by-now classical works, the description of dynamics on these homogeneous spaces for a general acting group remains a challenge.

\bigskip
		
\begin{center}
* 
\nopagebreak
*    *
\end{center}

We now proceed with introducing the notation needed to state our results. In the sequel, the meaning of the following groups, spaces, measures etc.~ will be fixed unless otherwise stated. Let $G$ be a semisimple real Lie group with finite centre and $Q<G$ a parabolic subgroup. Let $R_0 \unlhd Q$ be a normal algebraic subgroup and $R<Q$ be a closed subgroup containing $R_0$ such that $S:=Q/R_0$ is semisimple with finite centre and without compact factors and $\Lambda:=R/R_0$ is a discrete subgroup of $S$. We denote by $X$ the quotient space $G/R$ which will serve as the ambient space. A guiding example is provided by the homothety classes of rank-$k$ lattices in $\R^d$, see Example \ref{ex.ss} for a detailed description of these groups in that case.

All probability measures considered in this article will be Borel probability measures. We will denote by $\mu$ a probability measure on $G$. A measure $\mu$ on $G$ is said to have finite first moment if for a (equivalently any) irreducible finite-dimensional faithful linear representation $V$ of $G$, we have $\int \log \|g\|d\mu(g)<\infty$, where $\|.\|$ any choice of an operator norm on $\Endo(V)$. The group generated by the support of $\mu$ will be denoted $\Gamma_\mu$ and $H$ will denote the Zariski-closure of $\Gamma_\mu$ --- we will simply say that $\mu$ is Zariski-dense in $H$. Recall that a measure $\nu$ on $X$ is said to be $\mu$-stationary if it satisfies $\mu \ast \nu=\int g_\ast \nu d\mu(g)=\nu$, where $g_\ast \nu$ is the pushforward of $\nu$ by $g \in G$. By a stationary measure, we will understand a stationary probability measure. A $\mu$-stationary measure is said to be ergodic if it is an extremal point in the compact convex set $P_\mu(X)$ of $\mu$-stationary measures on $X$. Finally,  we will always suppose that $H$ is isomorphic to either $\SL_2(\R)$ or $\PGL_2(\R)$. The intersection of $H$ and the parabolic group $Q$ will be denoted $Q_H$.

\begin{wrapfigure}{r}{4cm}\label{figure.bundle}
 \begin{tikzpicture}

\coordinate[label = above :{${}$}] (0) at (-0.7,0);

\coordinate[label = above :{$X\simeq G/R$}] (1) at (1, 1.8);
    
    \coordinate (2) at (1, 1.8);
    
    \coordinate (3) at (1, 0);
    
    \coordinate[label = below :{$G/Q$}] (4) at (1, 0);
    
    \coordinate[label = right :{$\simeq S/\Lambda$}] (5) at (1.2, 1);

    \draw[->] (2) -- (3);
    
\end{tikzpicture}
\end{wrapfigure}

Before proceeding, in order to conceptually expose our results, we discuss the fibre bundle structure and various possible situations that arise; see the guiding Figure \ref{figure.cases}. Since the factored-out subgroup $R$ is contained in the parabolic $Q$, the space $X$ has a natural $G$-equivariant projection $\pi$ onto the flag variety $G/Q$. The fibres of $\pi$ are copies of the quotient $Q/R$, and by construction, we have $Q/R \simeq (Q/R_0)/(R/R_0)=S/\Lambda$. Since this projection is $G$, and hence $\Gamma_\mu$-equivariant, any $\mu$-stationary measure $\nu$ on $X$ projects down to a $\mu$-stationary measure $\overline{\nu}:=\pi_\ast \nu$ on $G/Q$. It follows that a first rough classification of stationary measures is provided by the classification in the base $G/Q$. Thanks to the results of Guivarc'h--Raugi \cite{guivarch-raugi} and Benoist--Quint \cite{BQ.compositio} (see \S \ref{subsub.stat.meas.base}), there are two types of projections giving rise to \textbf{Case 1} and \textbf{Case 2}, respectively, Dirac measures and Furstenberg measures on the base. In Case 1, the works of Benoist--Quint \cite{BQ2,BQ3} and Eskin--Lindenstrauss \cite{eskin-lindenstraus.short,eskin-lindenstrauss.long} directly apply and hence we will not comment on it further here (see \S \ref{subsub.dirac.base}).

If a stationary measure $\nu$ is in Case 2, up to replacing $Q$ by a conjugate, $Q_H$ is a parabolic subgroup of $H$ and the projection $\overline{\nu}$ of $\nu$ is the Furstenberg measure on $\calC:=H/Q_H$ in $G/Q$ (see \S \ref{subsub.furstenberg.base}). We will denote this projection as $\overline{\nu}_F$. In this case, the group $H$ preserves a subbundle of $X$, namely $\pi^{-1}(H/Q_H)$ which we will denote as $X_\mathcal{C}$ for brevity. 


\begin{figure}[H]\label{figure.cases}
 \begin{tikzpicture}

    \coordinate (O) at (0,0);

    \coordinate (1) at (-5,0);
    
    \coordinate[label = right : \small{\textbf{Case 1:} Trivial base $Q_H=H \longrightarrow$ Benoist--Quint, Eskin--Lindenstrauss}] (2) at (-4, 2.2);
    
    \coordinate (3) at (-2.5, 0);
  
    \coordinate[label = {\small$\underset{\text{\small{$Q_H<H$  parabolic}}}{\text{\textbf{Case 2:}}}$}] (10) at (-3.25, -1);
    
    \coordinate[label = right: {\small \textbf{Case 2.1:} $Q_H^o<R_0$: (Decomposable) Trivial fibre action: Prop.~ \ref{prop.trivial.fibre.measure.class}.}] (4) at (-0.5, 1.3);
    \coordinate[label = right: {\small \textbf{Case 2.2:} $Q_H^o \cap R_0 =R_u(Q_H^o)$: Diagonal fibre action: Theorem \ref{thm.measure.class.geod}.}] (5) at (-0.5, 0.35);
    
    \coordinate[label = right:] (6) at (-0.5, -1.7);
    \coordinate[label = right: {\small  $\underset{\text{\small Theorem \ref{thm.irreducible.H.decompsable}.}}{\text{\textbf{Case 2.3.a:} Irreducible $H$ $\longrightarrow$ Decomposable action }}$}] (7) at (2, -0.8);
    \coordinate[label = right: \small {\textbf{Case 2.3.b:} Example \ref{ex.to.be.treated}.}] (8) at (2, -2.2);
    
    \coordinate[label = {\small $\underset{\text{\small{$Q_H^o \cap R_0 =\{\id\}$}}}{\text{\textbf{Case 2.3:}}}$}] (12) at (-1.5, -2.5);


    \draw[-{Latex[length=2.7mm, width=1.3mm]}] (1) -- (2);
    \draw[-{Latex[length=2.7mm, width=1.3mm]}] (1) -- (3);
    \draw[-{Latex[length=2.7mm, width=1.3mm]}] (3) -- (4);
    \draw[-{Latex[length=2.7mm, width=1.3mm]}] (3) -- (5);
    \draw[-{Latex[length=2.7mm, width=1.3mm]}] (3) -- (6);
    \draw[-{Latex[length=2.7mm, width=1.3mm]}] (6) -- (7);
    \draw[-{Latex[length=2.7mm, width=1.3mm]}] (6) -- (8);

    \node at (1)[circle,fill,inner sep=1.2pt]{};
    \node at (2)[circle,fill,inner sep=1.2pt]{};
    \node at (3)[circle,fill,inner sep=1.2pt]{};
    \node at (4)[circle,fill,inner sep=1.2pt]{};
    \node at (5)[circle,fill,inner sep=1.2pt]{};
    \node at (6)[circle,fill,inner sep=1.2pt]{};
    \node at (7)[circle,fill,inner sep=1.2pt]{};
    \node at (8)[circle,fill,inner sep=1.2pt]{};

\end{tikzpicture}
\caption{List of all possible cases for $H \acts X_\mathcal{C}$}
\end{figure}

\vspace{-0.35cm}

A convenient way (which we will follow) to read the stationary measures and the action on $X_\mathcal{C}$ is to choose natural Borel trivializations of the bundle $X_\mathcal{C}$. By working with a class of trivializations (those induced by sections $G/Q \to G$) which we call standard trivializations (see \S \ref{sec.prelim}), we will consider the identifications $X \simeq G/Q \times S/\Lambda$ and $X_\mathcal{C} \simeq H/Q_H \times S/\Lambda$ where the latter identification is made equivariant by $H$-acting on the right-hand-side via a cocycle $\alpha: H \times H/Q_H \to S$.

Now, according to the algebraic relations between $Q_H$ and the group $R_0 \unlhd Q$, we distinguish three (exhaustive) possibilities giving rise to different dynamics on fibres via the cocycle $\alpha$. \textbf{Case 2.1} is the case when $Q_H^o<R_0$. In this situation, we have trivial dynamics on the fibre and every stationary measure on $X$ is a copy of the Furstenberg measure (see \S \ref{subsec.case.2.1}). \textbf{Case 2.2} is when $Q^o_H \cap R_0$ is neither $Q^o_H$ nor $\{\id\}$. Since $R_0$ is a normal subgroup of $Q$, the intersection $Q^o_H\cap R_0$ must be the unipotent radical $R_u(Q_H^o)$. In this situation, up to a judicious choice of trivialization, the cocycle $\alpha$ takes values in a rank-one diagonal subgroup of $S$ and we obtain a classification of stationary measures (Theorem \ref{thm.measure.class.geod}) as product measures on $H/Q_H \times S/\Lambda$ in the second factor invariant under diagonal flow. This is the result for which we develop our drift argument and also give an alternative proof, under a stronger moment assumption, using a uniform quantitative renewal theorem for random matrix products. Finally, the remaining \textbf{Case 2.3} occurs when $Q_H^o \cap R_0=\{\id\}$. In this case we restrict the ambient group $G$ to be $\SL_n(\R)$ or $\PGL_n(\R)$. When the associated linear or projective $H$-action is irreducible (\textbf{Case 2.3.a}), we prove that the cocycle $\alpha$ comes from an algebraic morphism $H \to S$ (what we call a decomposable action, see \S \ref{subsub.bundle.dec}) allowing us to reduce the analysis to the work of Benoist--Quint and Eskin--Lindenstrauss again (Theorem \ref{thm.irreducible.H.decompsable}). Our result for this case allows an interpretation of the aforementioned phenomenon appearing in the work of Sargent--Shapira \cite{sargent-shapira} and generalizes it. Finally,  \textbf{Case 2.3.b} occurs when $H$ is reducible. The description of dynamics in this case remains open, we provide a conjecture (Conjecture \ref{conjecture}) expressing our expectation.

\bigskip

We now state our measure classification results in Case 2.2 and Case 2.3.a followed by the corresponding equidistribution results.


\begin{theorem}[Case 2.2:~Diagonal flow invariance and product structure]\label{thm.measure.class.geod}
Let the space $X$ and groups $G,Q,R_0,R,S\simeq Q/R_0,\Lambda\simeq R/R_0$ and $H$ be as defined above. Suppose that $Q_H<H$ is a parabolic subgroup and $Q_H^o \cap R_0=R_u(Q_H^o)$. Let $\mu$ be a Zariski-dense probability measure on $H$ with finite first moment. Then, there exist a standard trivialization $X_\mathcal{C} \simeq H/Q_H \times S/\Lambda$ and a one-dimensional connected diagonal subgroup $D$ of $S$ satisfying the following. Let $\nu$ be a $\mu$-stationary and ergodic probability measure on $X_\calC$. Then, there exists a $D$-invariant and ergodic probability measure $\tilde{\nu}$ on $S/\Lambda$ such that we have $\nu=\overline{\nu}_F\otimes \tilde{\nu}$.
\end{theorem}

The hypotheses of this result entail a lack of expansion on the fibres whose existence is a key feature exploited in  Benoist--Quint's exponential drift argument. Instead, we adapt a drift argument inspired by the works of Eskin--Mirzakhani \cite{eskin-mirzakhani} and Eskin--Lindenstrauss \cite{eskin-lindenstraus.short} that exploits the interaction between different fibres. The commutativity of the target group of the cocycle considerably simplifies the steps compared to the previous works \cite{eskin-mirzakhani,eskin-lindenstraus.short}. Moreover, if $\mu$ is supposed to have a finite exponential moment, taking advantage of the special setting, we give an alternative proof using uniform quantitative renewal theorem due to Li \cite{jialun.ens} and Li--Sahlsten \cite{jialun.advances}. We defer any further comments to the more detailed discussion below in \S \ref{subsec.intro.proofs}.

\begin{remark}\label{rk.conversely.to.diagonal.thm}
The converse to Theorem \ref{thm.measure.class.geod} is also true in the sense that when $Q_H<H$ is a parabolic subgroup such that $Q_H^o \cap R_0=R_u(Q_H^o)$, there exists a standard trivialization $X_\mathcal{C} \simeq H/Q_H \times S/\Lambda$ and  an index-two extension $D^{\pm}$ of $D$ such that for any $D^\pm$-invariant probability $\tilde{\nu}$, the measure $\overline{\nu}_F \otimes \tilde{\nu}$ is $\mu$-stationary probability measure on $X$.
\end{remark}

We now continue with Case 2.3.a. We first introduce the following definition to state our result.

\begin{definition}\label{def.decomposable}
An $H$-homogeneous subbundle $X_\mathcal{C}$ of $X$ is said to be \textbf{decomposable} 
if $X_\mathcal{C}$ is isomorphic as $H$-space to $\mathcal{C} \times S/\Lambda$,
where the latter is endowed with the $H$-action $h(c,f)=(hc,\rho(h)f)$ and $\rho:H \to S$ is a morphism extending $Q_H \hookrightarrow Q \twoheadrightarrow S
$.
\end{definition}

Our next measure classification result\footnote{After we have obtained our results, we have been informed by Uri Shapira that in a sequel work to \cite{sargent-shapira} together with Uri Bader and Oliver Sargent, for the classification of stationary measures, they independently obtain the same result (and also introduce a similar notion as in Definition \ref{def.decomposable}). Our proof ideas for this case seem to be similar. They also obtain equidistribution results in some situations of Case 2.3.a for random walks starting outside the bundle $X_\mathcal{C}$. We thank Uri Shapira for related and kind discussions.} is the following.

\begin{theorem}[Decomposable action]\label{thm.irreducible.H.decompsable}
Let $G=\PGL_n(\R)$, the space $X$ and groups $Q,R_0,R,S\simeq Q/R_0,\Lambda\simeq R/R_0$ and $H$ be as defined above. 
Suppose that the $H$-action on $\P(\R^{n})$ is irreducible. Then, there exists a unique $H$-compact orbit $\calC$ in $G/Q$ and the $H$-action on $X_\mathcal{C}$ is decomposable. In particular, given a Zariski-dense probability measure $\mu$ on $H$ with finite first moment, we have a bijection
\begin{equation}\label{eq.bijection.in.thm.SL2.dec}
    P_\mu^{\erg}(X_{\mathcal{C}}) \simeq  P_\mu^{\erg}(S/\Lambda),
\end{equation}
where the action of $\mu$ on $S/\Lambda$ comes from the morphism $\rho: H \to S$ in Definition \ref{def.decomposable}.
\end{theorem}


This result provides, in a more general setting, a conceptual explanation for the existence of the invariant section discovered by Sargent--Shapira (see the relevant discussion in \cite{sargent-shapira}) and allows us to deduce affirmative answers to (1),(2) and (6) \cite[Problem 1.13]{sargent-shapira}.

\begin{remark}
1. Case 2.3.a is the main particular case of this theorem. In the above statement, if we suppose that $R_0 \neq Q$, then one can verify that $Q_H^\circ\cap R_0=\{\id\}$ and we are in Case 2.3.a.\\
2. It might be possible to generalize the setting of the above theorem to the case where $G$ is a simple $\R$-split linear Lie group and $H<G$ is the image of a principal $\SL_2(\R)$ in $G$ in the sense of Kostant \cite{kostant}.
\end{remark}

In view of the measure classification results of Benoist--Quint and Eskin--Lindenstrauss \cite[Theorem 1.3]{eskin-lindenstrauss.long} on quotients by discrete subgroups, i.e.~ the right-hand-side of \eqref{eq.bijection.in.thm.SL2.dec}, the following is a consequence of Theorem \ref{thm.irreducible.H.decompsable} (and Proposition \ref{prop.decomposable.measure.class}). Recall that a homogeneous measure $\tilde{\nu}$ on $S/\Lambda$ is a probability measure supported on a closed orbit of its stabilizer $S_0<S$. We also say that such a measure is $S_0$-homogeneous. 

\begin{corollary}[Homogeneous fibres]\label{corol.homogeneous.fibres}
Keep the setting of Theorem \ref{thm.irreducible.H.decompsable}. Let $\nu$ be a $\mu$-stationary and ergodic probability measure on $X_\mathcal{C}$. There exists a trivialization $X_\mathcal{C} \simeq H/Q_H \times S/\Lambda$ in whose coordinates $\nu$ is a product measure $\overline{\nu}_F \otimes \tilde{\nu}$, where $\tilde{\nu}$ is $S_0$-homogeneous.
\end{corollary}

\begin{remark}
Consider any standard trivialization 
$X_\mathcal{C} \simeq H/Q_H \times S/\Lambda$. Let $\nu=\int \delta_\theta \otimes \nu_\theta d\overline{\nu}_F(\theta)$ be the disintegration of $\nu$ over the base $H/Q_H$. Then, there exists a closed subgroup $S_0<S$ such that for $\overline{\nu}_F$-a.e.~$\theta \in H/Q_H$, the fibre measure $\nu_\theta$ is $S_\theta$-homogeneous, where $S_\theta$ is a conjugate of $S_0$.
\end{remark}

\bigskip

The last remaining possibility for the action of $H$ on an $H$-invariant subbundle $X_\mathcal{C}$ is Case 2.3.b, it happens when $Q_H$ is a parabolic subgroup of $H$ and $Q_H^o \cap R_0=\{\id\}$ but $H$-action on $\P(\R^n)$ is not irreducible (see Example \ref{ex.to.be.treated}). In the statement below, we conjecture that the fibre measures are homogeneous, supposing only that the natural morphism $Q_H \to S$ has finite kernel (equivalently, $Q_H^o \cap R_0=\{\id\}$); in other words, the conclusion of Corollary \ref{corol.homogeneous.fibres} holds, without the irreducibility assumption. 

\begin{conjecture}[Homogeneous fibres]\label{conjecture}
Let $G=\PGL_n(\R)$, the space $X$ and groups $Q,R_0,R,S\simeq Q/R_0,\Lambda\simeq R/R_0$ and $H$ be as defined above. Suppose we are in Case 2.3, i.e.~ $Q_H$ is a parabolic subgroup of $H$ and $Q_H^o \cap R_0=\{\id\}$. Then the conclusion of Corollary \ref{corol.homogeneous.fibres} holds.
\end{conjecture}

We now turn to the equidistribution aspect of random walks on $H$-subbundles $X_\mathcal{C}$ of $X$. We keep the same setting as in the measure classification part above; we suppose in addition that $\Lambda$ is a lattice in $S$ (except for Theorem \ref{thm.equidist.geod} below). Let as usual $\mu$ be a probability measure on $G$ that is Zariski-dense in a copy $H$ of $\SL_2(\R)$ or $\PGL_2(\R)$ in $G$. Given a point $x \in X_\mathcal{C}$, we are interested in describing the asymptotic behaviour of the averaged distribution $\frac{1}{n} \sum_{k=1}^n \mu^{\ast n}\ast \delta_x$ of the random walk on $X$ starting from $x$ up to the step $n$.

In all cases in which we treat the measure classification problem, i.e.~ all cases except Case 2.3.b, it will be possible to address the equidistribution problem. In fact, Case 1 (trivial base) is precisely the setting of Benoist--Quint \cite{BQ3} so the corresponding equidistribution results (see \cite{prohaska-sert-shi, benard-desaxce} extending the original results with respect to moment hypotheses) directly apply; we do not comment on it further here. Case 2.1 boils down to the equidistribution to the Furstenberg measure; even quantitative statements are known for this case, see \S \ref{subsub.equidist.trivial.fibre}. Finally, thanks to the decomposability obtained in Theorem \ref{thm.irreducible.H.decompsable}, it is not hard to see that Case 2.3.a also boils down to the setting of Benoist--Quint, see Proposition \ref{prop.equidist.H.irred}. 

Therefore the only case that needs to be handled is Case 2.2, i.e.~diagonal fibre action. In this case, we will observe that (see Lemma \ref{lemma.its.iwasawa}) we have a standard trivialization $X_\calC\simeq \calC\times_\alpha S/\Lambda$ such that the action $\alpha$ of $H$ on the fibre $S/\Lambda$ is by a one-dimensional diagonal subgroup $D$ of $S$ through the Iwasawa cocycle $\sigma$ up to a sign. It is  well-known that the $D$-orbits of different points on $S/\Lambda$ can exhibit very different statistical behaviours, i.e.~ not characterized by a single $D$-invariant measure. Given the existence of this chaotic behaviour, the most one can hope to establish is that the statistical behaviour of the $\mu$-random walk in the fibres matches that of the $D$-flow. This is the content of the following result. In the statement, the equidistribution of a $D$-orbit is understood with respect to a Haar/Lebesgue measure on $D$.

	
\begin{theorem}[Diagonal fibre action: equidistribution]\label{thm.equidist.geod}
Keep the hypotheses and notation of Theorem \ref{thm.measure.class.geod} and let $X_\mathcal{C} \simeq H/Q_H \times S/\Lambda$ be the trivialization given by Theorem \ref{thm.measure.class.geod}. Suppose in addition that the measure $\mu$ has finite exponential moment and $\Gamma_\mu$ is inside the connected component of $H\simeq\PGL_2(\R)$. Then, the $D$-orbit of $z \in S/\Lambda$ equidistribute to a probability measure $m$ on $S/\Lambda$ if and only if for any $x=(\theta,z)\in X_\calC$, we have the convergence
\[ \frac{1}{n}\sum_{k=1}^n\mu^{*k}*\delta_x\rightarrow \bar{\nu}_F\otimes m  \quad  \text{as} \; \; n \to \infty. \]
\end{theorem}

\begin{remark}
For $H\simeq \SL_2(\R)$, we have a similar equidistribution result. But the statement is more complicated. See \S\ref{subsec.equidist.diagonal} for more details.
\end{remark}

\begin{remark}[Alternative proof for Theorem \ref{thm.measure.class.geod}]
The results we establish to prove Theorem \ref{thm.equidist.geod} allow us to obtain a different proof of Theorem \ref{thm.measure.class.geod} under the additional finite exponential moment condition: let $\nu$ be a $\mu$-stationary and ergodic measure. By Chacon-Ornstein ergodic theorem, there exists $x$ such that $\frac{1}{n}\sum_{k=1}^n\mu^{*k}*\delta_x\rightarrow \nu$ as $n \to \infty$. From the proof of Theorem \ref{thm.equidist.geod}, we actually obtain that $\frac{1}{n}\sum_{k=1}^n\mu^{*k}*\delta_x$ and $\bar\nu_F\otimes \frac{1}{t}\int_0^t \delta_{\alpha(t)(z)}\ dt$ have the same limit as $t,n \to \infty$, where $\alpha(t)$ is the flow of $D$ and $x=(\theta,z)$. Therefore $\bar\nu_F\otimes \frac{1}{t}\int_0^t \delta_{\alpha(t)(z)}\ dt\to\nu$, which implies $\frac{1}{t}\int_0^t \delta_{\alpha(t)(z)}\ dt\to m$ for some $D$-invariant probability measure $m$ and hence the conclusion of Theorem \ref{thm.measure.class.geod}.
\end{remark}

\subsection{Ideas of proofs}\label{subsec.intro.proofs}
Before finishing the introduction, we give a brief overview of the ideas of proofs used to obtain the main results of this paper.\\[-4pt]

\textbullet ${}$ \textbf{Case 2.2 (measure classification): Drift argument.} The basic idea in Case 2.2 is to use the non-triviality of the fibre bundle structure of $X_\mathcal{C}$ over $\mathcal{C}$ to obtain invariance of the measures.
More concretely, for each cocycle $\alpha:H\times \calC\to S$, one can try to define a cross-ratio for quadruple $a,a',b,b'\in H^\N$
\[C_\alpha(a,a',b,b')=\lim_{n,m\to \infty}\alpha(a'^n,\xi(b))\alpha(a^{m},\xi(b'))^{-1}\alpha(a'^n,\xi(b'))^{-1}\alpha(a^m,\xi(b')) \]
with suitable limits of $n,m$, where $\xi$ is some map from $H^\N$ to $\calC$. If the cocycle $\alpha$ is cohomologous to a morphism from $H$ to $S$, that is, the bundle structure is trivial, then any reasonable definition of cross-ratio will yield no information. This  corresponds to \textit{decomposable} action in Theorem \ref{thm.irreducible.H.decompsable}. Otherwise, if the bundle structure is not trivial, as in Case 2.2, then the cross-ratio is non-trivial for generic four points $a,a',b,b'$ and yields certain information on the relation between asymptotic behaviour of products corresponding to those points. In this case, we adapt the drift argument of \cite{eskin-lindenstraus.short} to ``six points drift argument'' to exploit this information and obtain invariance under a limit cross-ratio. This six points drift argument is very different from the drift argument in \cite{BQ1} or \cite{sargent-shapira}; we do not use expansion on some tangent directions (indeed, in Case 2.2, we have no expansion). It is really the non-triviality of the cross-ratio or equivalently the bundle structure that helps us to obtain the invariance of the measures.

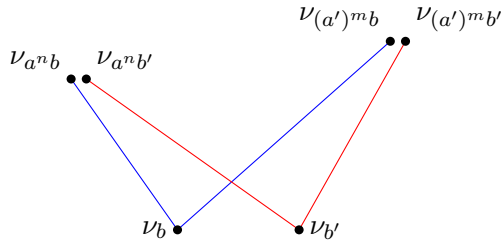
\begin{figure}[H]\label{fig.drift}
 \begin{tikzpicture}

    \coordinate (O) at (0,0);

    \coordinate[label = left:$\nu_b$] (1) at (-0.8,-1.2);
    
    \coordinate[label = right:$\nu_{b'}$] (2) at (0.8, -1.2);

    \coordinate[label = above left:$\nu_{a^nb}$] (3) at (-2.2,0.8);

    \coordinate[label = above right:$\nu_{a^nb'}$] (4) at (-2,0.8);

    \coordinate[label = above left:$\nu_{(a')^m b}$] (5) at (2,1.3);
    
    \coordinate[label = above right:$\nu_{(a')^m b'}$] (6) at (2.2,1.3);

    \draw[blue] (1) -- (3);

    \draw[blue] (1) -- (5);

    \draw[red] (2) -- (4);
    
    \draw[red] (2) -- (6);

    \node at (1)[circle,fill,inner sep=1.2pt]{};
    \node at (2)[circle,fill,inner sep=1.2pt]{};
    \node at (3)[circle,fill,inner sep=1.2pt]{};
    \node at (4)[circle,fill,inner sep=1.2pt]{};
    \node at (5)[circle,fill,inner sep=1.2pt]{};
    \node at (6)[circle,fill,inner sep=1.2pt]{};

\end{tikzpicture}
\caption{Six points drift argument}
\label{tikzpic}
\end{figure}

\vspace{-9pt}

\textbullet ${}$ \textbf{Case 2.3 (measure classification): Decomposable action.} In some sense, the key difficulty in this case  is to suspect the possibility of the existence of a decomposable action in our setting. Once one has this possibility in mind, one can use ideas about cocycles going back to Mackey \cite{mackey49, mackey58} (see Varadarajan \cite[\S 5]{varadarajan.quantum} or Zimmer \cite{zimmer.book} for precise expressions) to establish this decomposability. The latter expresses a certain algebraic structure in, or equivalently triviality of, the fibre-bundle $X_\mathcal{C}$ and in more concrete terms it boils down to an extension of a natural morphism $Q_H \to S$ to a larger group (only possibility being $H$ in our setting). Once this is established, one reduces the situation to a trivial-bundle structure and hence one can bring in the results of Benoist--Quint and Eskin--Lindenstrauss.\\[-3pt]

\textbullet ${}$ \textbf{Case 2.2 (Equidistribution): Uniform quantitative renewal theorem.}
The key point that enables us to obtain the equidistribution result (Theorem \ref{thm.equidist.geod}) and an alternative proof of Theorem \ref{thm.measure.class.geod} under a stronger moment assumption, is the fact that thanks to the particular situation we have precise control of random matrix products in the form of a uniform quantitative renewal theorem and exponential large deviation estimates. More precisely, under suitable trivialization, the Ces\`{a}ro average can be expressed as
\[\frac{1}{n}\sum_{k=1}^n\mu^{*k}*\delta_x=\frac{1}{n}\sum_{k=1}^n\int \delta(g\theta,\alpha(\sigma_\chi(g,\theta))z) \ d\mu^{*k}(g), \]
where $x=(\theta,z)$ and $\sigma_\chi$ is some cocycle from $H\times \calC$ to $\R$. This is very similar to the renewal sum $\sum_{k=1}^\infty \int \delta(g\theta,\sigma_\chi(g,\theta)-t)\ d\mu^{*k}(g)$ which converges to the product measure $\bar{\nu}_F\otimes Leb_{\R^+}$ with respect to compactly supported continuous functions. Combined with exponential large deviation estimates and good error estimates from the uniform quantitative renewal theorem, we can prove the equidistribution result.

\bigskip

This article is organized as follows. Section \ref{sec.prelim} contains some preliminary tools about fibred dynamics, cocycles and stationary measures. Section \ref{sec.meas.class} is devoted to proving the measure classification results; Theorem \ref{thm.measure.class.geod} and Theorem \ref{thm.irreducible.H.decompsable} are proved therein. Finally, Section \ref{sec.equidist} contains the equidistribution results, in particular the proof of Theorem \ref{thm.equidist.geod}.



\section{Preliminaries: Cocycles, decomposable actions and stationary measures}\label{sec.prelim}

This section contains a collection of  preliminaries for the proofs in the following parts. We adopt a general setting. In \S \ref{subsec.cocycles}, after a discussion of cocycle induced by trivializations or sections, we introduce the notion of a decomposable action and present an important criterion for decomposability. \S \ref{subsec.stat.measures} contains a discussion of stationary measures and their decompositions, due to Furstenberg. Finally, in \S \ref{subsec.meas.class.sec2}, we single out a description of stationary measures for decomposable actions.

\subsection{Generalities on cocycles and decomposable actions}\label{subsec.cocycles}
Let $G$ be a locally compact and second countable (lcsc) group and $X$ a standard Borel space endowed with a Borel $G$-action $G \times X \to X$. Let $Q<G$ be a closed subgroup and suppose we have a measurable $G$-equivariant surjection $\pi: X \to G/Q$. We shall refer to such a $G$-space $X$ as a fibre bundle over $G/Q$. A fruitful way to describe the $G$-action on such bundles $X$ is by using the notion of cocycles. This approach -- going back to the work of Mackey \cite{mackey49, mackey58} (see Varadarajan \cite[\S 5]{varadarajan.quantum}) on induction of unitary representations -- will be instrumental to our considerations. 

\subsubsection{Cocycles defined by actions and vice versa}

Given a bundle over $G/Q$ with $G$-action, let $F$ be a copy of the fibre above $Q$, i.e.~ of the Borel set $\pi^{-1}(Q)$ endowed with the $Q$-action. A $(G,Q)$-bundle trivialization of $X$ is a Borel isomorphism $\phi=(\phi_1,\phi_2):X \simeq G/Q \times F$ such that $\phi_1$ is $G$-equivariant and $\phi_2$ is equivariant with respect to a $Q$-valued cocycle $\alpha: G \times G/Q \to Q$, i.e.~ $\phi(gx)=(g\phi_1(x),\alpha(g,\phi_1(x))\phi_2(x))$. Notice that $F$ has a natural $Q$-action. Recall that a cocycle $\alpha:G \times G/Q \to Q$ is a map satisfying $\alpha(g_1g_2,f)=\alpha(g_1,g_2f) \alpha(g_2,f)$  for every $g_1,g_2 \in G$ and $f \in G/Q$  (this corresponds to what is called a strict cocycle in \cite{varadarajan.quantum}). By using the cocycle relation, one sees that any cocycle $\alpha: G \times G/Q \to Q$ endows the space $G/Q \times F$ with a $G$-action. We shall denote the space $G/Q \times F$ endowed with a $G$-action induced by a $Q$-valued cocycle $\alpha$ by $G/Q \times_\alpha F$. Therefore, a $(G,Q)$-bundle trivialization is a $G$-equivariant isomorphism between $X$ and $G/Q \times_\alpha F$ for some $Q$-valued cocycle $\alpha$.

In the rest of this paper, the ambient space $X$ which we will work with will be, in particular, a homogeneous space of a lcsc group $G$. For $x \in X$, we denote by $G_x$ the stability group $\{g \in G : gx=x\}$ 
so that we have a $G$-equivariant identification $X \simeq G/G_x$. We will suppose that the stability group $G_{x_0}=:R$ of a base point $x_0 \in X$ is contained in $Q$ so that we have a continuous $G$-equivariant surjection $\pi: G/R \simeq X \to G/Q$ turning the homogeneous space $X$ into a bundle over $G/Q$ with $G$-action. The choice of the base point $x_0$ identifies the fibre $\pi^{-1}(Q)$ with the $Q$-homogeneous space $Q/R$. In this setting, any Borel section $s:G/Q \to G$ yields a trivialization of $X$ given by the Borel isomorphism
\begin{equation}\label{eq.trivialization}
\begin{aligned}
X &\to G/Q \times Q/R\\
x &\mapsto (\pi(x), s(\pi(x))^{-1}x).
\end{aligned}
\end{equation}
Here the Borel section $s$ is a section of the principal $Q$-bundle $G\rightarrow G/Q$, i.e.~ $s(gQ)Q=gQ$.
Such a section $s$ induces a Borel cocycle $\alpha:G \times G/Q \to Q$ by setting
\begin{equation}\label{eq.construct.cocycle}
\alpha(g,hQ):=s(ghQ)^{-1}g s(hQ),
\end{equation}
which makes the trivialization \eqref{eq.trivialization} into a $(G,Q)$-bundle trivialization. 

For a $G$-homogeneous bundle $X$, we will say that a $(G,Q)$-bundle trivialization $X \simeq G/Q \times_\alpha Q/R$ is \textit{standard} if the morphism $\rho_\alpha: Q \to Q$ given by $\rho_\alpha(q)=\alpha(q,Q)$ is conjugate to the identity morphism $Q \to Q$. A trivialization given by a choice of section as above is standard. Conversely, a standard trivialization is induced by a choice of section $s:G/Q \to G$.

In the sequel, we will assume further structure on the groups making up the space $X$. Namely, we will suppose that there exists a closed normal subgroup $R_0$ of $Q$ which is also a subgroup of $R$ so that writing $S:=Q/R_0$ and $\Lambda=R/R_0$, we have an identification of the fibre $Q/R$ with $S/\Lambda$. The reason for this is that starting from Section \ref{sec.meas.class}, the group $\Lambda$ will be a discrete subgroup of $S$ which will make the action on $Q/R \simeq S/\Lambda$ more tractable. Composing a cocycle $\tilde{\alpha}:G \times G/Q \to Q$ with the epimorphism $\pi_1:Q \to S=Q/R_0$, we obtain an $S$-valued cocycle $\alpha: G \times G/Q \to S$. 
Since the action of $Q$ on the fibre $Q/R$ factors through $S$, the $S$-valued cocycle $\alpha=\pi_1\circ\tilde{\alpha}$ is sufficient to reconstruct the bundle. From now on, we will mainly consider $S$-valued cocycles.

\subsubsection{Equivalence of cocycles}\label{subsub.equiv.cocycles}
Let $G'$ be a lcsc group. Two $G'$-valued cocycles $\alpha, \beta: G \times G/Q \to G'$ are said to be equivalent, denoted $\alpha \sim \beta$, if there exists a Borel map $\phi:G/Q \to G'$ such that for every $x \in G/Q$ and $g \in G$, we have 
\begin{equation}\label{eq.eq.cocycles}
\alpha(g,x)=\phi(gx)^{-1} \beta(g,x) \phi(x).    
\end{equation}
It is clear that for two $S$-valued cocycles $\alpha$ and $\beta$ over $G/Q$, if $\alpha \sim \beta$, then the associated $G$-spaces $G/Q \times_\alpha Q/R$ and $G/Q \times_\beta Q/R$ are isomorphic. Accordingly, $S$-valued cocycles obtained from different sections $s:G/Q \to G$ via \eqref{eq.construct.cocycle} are equivalent. 

In the sequel, we will be interested in actions of subgroups $H$ of $G$ on a $G$-homogeneous bundle $X$ and the associated subbundles. We will use similar terminology for cocycles restricted to $H$. Let $H$ be a closed subgroup of $G$ and $\mathcal{C} \subseteq G/Q$ an $H$-homogeneous closed subset of $G/Q$. Up to replacing $Q$ by a conjugate, suppose $\mathcal{C}=HQ$ so that $\mathcal{C} \simeq H/Q_H$ where $Q_H:=H \cap Q$. In that case $X_\mathcal{C}:=\pi^{-1}(\mathcal{C}) \subseteq X$ is an $H$-invariant closed (not necessarily $H$-homogeneous) subset of $X$ giving rise to a bundle over $\mathcal{C}$ with $H$-action. 

\begin{example}\label{ex.ss} The above situation appears in (and is motivated by) the setting of the work of Sargent--Shapira \cite{sargent-shapira}: Let $X$ be the set of homothety-equivalence classes of rank-2 lattices in $\mathbb{R}^3$. The set $X$ has a natural lcsc topology and the group $G=\SL_3(\R)$ acts continuously and transitively on $X$. The connected component $R_0$ of the stabilizer $R$ of a point $x \in X$ consists of the solvable radical of a maximal parabolic group $Q$ in $G$ 
and we have a surjective map $S = Q/R_0 \to Q/R\simeq S/\Lambda\simeq \SL_2(\R)/\SL_2(\Z)$, where $R/R_0 \simeq \Lambda=\PGL_2(\Z)$. One can then consider the action of the subgroup $H=\SO(2,1)<G$ on $X$ which has a unique minimal invariant subset $\mathcal{C}$ in $G/Q$. The group $Q_H=H \cap Q$ corresponds to a Borel subgroup of $H$. 
\end{example}

\subsubsection{Induced morphisms and decomposable subbundle actions}\label{subsub.bundle.dec}

Recall the notion of a decomposable bundle in Definition \ref{def.decomposable}. We provide a criterion to ensure that an $H$-homogeneous bundle $X_\mathcal{C}$ is decomposable. %


\begin{proposition}\label{prop.decomposable}
If the morphism $Q_H \hookrightarrow Q \twoheadrightarrow S
$ extends to a morphism $H \to S$, then $X_\mathcal{C}$ is decomposable.
\end{proposition}

The following statement is a version of \cite[Proposition 4.2.16]{zimmer.book} and provides a useful characterization of a decomposable action in our setting. 
Let $\Co(H,Q_H,S)$ denote the set of Borel cocycles $H \times H/Q_H \to S$. Given $\alpha \in \Co(H,Q_H,S)$, the map $\rho_\alpha$ defined by $\rho_\alpha(p):=\alpha(p,Q_H)$ for $p \in Q_H$ defines a Borel (hence continuous) morphism from $Q_H \to S$. 

The proof is based on an important observation of Mackey which characterizes equivalence classes of Borel cocycles $\Co(H,P,G')$ by conjugacy classes of induced morphisms $P \to G'$ where conjugacy is understood by an element of $G'$ in the target. We have the following result from \cite[Theorem 5.27]{varadarajan.quantum} that we adapt to our setting here.

\begin{lemma}\label{lemma.cocycle.bijection}
The map 
\begin{equation*}
\begin{aligned}
\Co(H,P,G') &\to \Hom(P,G')\\
\alpha & \mapsto \rho_\alpha
\end{aligned}
\end{equation*}
is a surjective map that descends to a bijection when we quotient $\Co(H,P,G')$ by equivalence of cocycles and $\Hom(P,G')$ by conjugation in $G'$.
\end{lemma}
\begin{proof}
Fix a Borel section $s:H/P \to H$ with $s(P)=\id$. Given a morphism $\rho:P \to G'$, the map given by 
\begin{equation}\label{eq.morphism.to.cocycle}
\alpha(h_1,hP)=\rho(s(h_1hP)^{-1}h_1s(hP))    
\end{equation}
is a cocycle whose restriction to $P \simeq P \times \{P\}$ recovers $\rho:P \to G'$. This shows that the map is surjective.

Note that if two cocycles $H \times H/P \to G'$ are equivalent, then it is clear that the morphisms $P \to G'$ that they induce are conjugate by an element of $G'$. In particular, the map $\alpha \mapsto \rho_\alpha$ descends to a (surjective) map on the equivalence classes of cocycles.

To show that this map is a bijection, let $\alpha$ and $\beta$ be two cocycles $H \times H/P \to G'$ and suppose that $\rho_\alpha(.)=l\rho_\beta(.)l^{-1}$ for some element $l \in G'$. By direct computation, we have
\begin{equation}\label{eq.section.inout}
\alpha(s(h_1hP),P)^{-1}\alpha(h_1,hP)\alpha(s(hP),P)=\rho_\alpha(s(h_1hP)^{-1}h_1s(hP))^{-1}.
\end{equation}
Writing the equivalent of \eqref{eq.section.inout} for $\beta$ (substituting $\rho_\beta$ for $\rho_\alpha$) and combining it with \eqref{eq.section.inout}, we get
$$
\alpha(h_1,hP)=\alpha(s(h_1hP), P)l^{-1}\beta(s(h_1hP),P)^{-1}\beta(h_1,hP)\beta(s(hP), P)l\alpha(s(hP),P)^{-1}.
$$
This shows that $\alpha \sim \beta$ via the fibre automorphisms given by the map $$hP \mapsto \beta(s(hP), P)l\alpha(s(hP), P)^{-1}$$ which proves the claim.
\end{proof}

We say that a cocycle $\alpha: H \times H/P \to S$ is of morphism-type with morphism $\rho$ if there exists a Borel morphism $\rho:H \to S$ such that $\alpha(h_1,hP)=\rho(h_1)$ for every $h_1,h \in H$.
We can now give proof of the decomposability criterion. 

\begin{proof}[Proof of Proposition \ref{prop.decomposable}]
Choose a Borel section $s: G/Q \to G$ with $s(Q)=\id$.
Let $\tilde\alpha:G \times G/Q \to Q$ be the associated cocycle. Recall $\pi_1$ is the quotient map $Q\to S=Q/R_0$. The associated morphism $\rho_{\tilde\alpha}:Q \to Q$ is then the identity map and hence the map $\pi_1\circ \rho_{\tilde\alpha}:Q_H\to Q\to S$ extends to a morphism $\tau:H \to S$ by assumption. We thus get a morphism-type cocycle $\beta:H \times H/Q_H \to S$ by taking $\beta(h,c)=\tau(h)$. By Lemma \ref{lemma.cocycle.bijection}, the cocycle $\alpha:=\pi_1\circ \tilde\alpha$ restricted to $H \times H/Q_H$ and $\beta$ are equivalent. Since equivalent cocycles induce isomorphic bundles, we are done.
\end{proof}

\subsection{Skew-product systems and stationary measures}\label{subsec.stat.measures}

Let $H^\Z$ be the set of two-sided sequences of elements of $H$. We denote an element (bi-infinite word) of $H^\Z$ by $w=(b,a)$ where, by convention, we consider $b=(\ldots, b_{-2},b_{-1}) \in H^{-\N^\ast}$ and $a=(a_0,a_1,\ldots) \in H^\N$. The sequence $b$, also denoted $w^-$ will be referred to as the past of $w$ and $a$, also denoted $w^+$, as the future of $w$. We denote by $T$ the shift map on $H^\Z$ taking one step forward to future, i.e.~ $T(b,a)=(ba_0,Ta)$, where $ba_0$ is the concatenation $(\ldots, b_{-1},a_0)$ and $Ta$ is the image of $a$ under the usual shift map, also denoted $T$, on $H^\N$. Accordingly, the inverse of $T$ is given by $T^{-1}(b,a)=(T^{-1}b,b_{-1}a)$, where $T^{-1}b=(\ldots, b_{-2})$.

Let $\mu$ be a probability measure on $H$ and $\nu$ a $\mu$-stationary probability measure on a locally compact and second countable $H$-space $Y$. It follows from the martingale convergence theorem that the limit as $n\to \infty$ of $b_{-1}\ldots b_{-n} \nu$ exists for $\mu^{\Z}$-almost every $w$; it will be denoted by $\nu_{w}$, or sometimes $\nu_b$. These limit measures satisfy a key equivariance property which says that for $\mu^{-\N^\ast}$-a.e. $b \in H^{-\N^\ast}$, we have $\nu_{T^{-1}b}=b_{-1}^{-1}\nu_b$ or equivalently, for $\mu^{-\N^\ast}$-a.e. $b \in H^{-\N^\ast}$ and $\mu$-a.e. $a_0 \in H$, we have $\nu_{ba_{0}}=a_0\nu_b$.

Stationary measures can also be seen as part of invariant measures on skew-product systems. Let $\hat{Y}$ denote the product $H^\Z \times Y$ and  $\hat{T}$ the skew-shift given by $\hat{T}((b,a),y)=(T(b,a), a_0 y)$. A basic fact (see e.g.~ \cite[Chapter 2]{bq.book}) is that given a probability measure $\mu$ on $H$, any $\mu$-stationary measure on $Y$ gives rise to a $\hat{T}$-invariant measure on $\hat{Y}$ that projects onto $\mu^\Z$ on the $H^\Z$ factor: indeed given a $\mu$-stationary measure $\nu$, the measure 
$$\hat{\nu}=\int \delta_{w} \otimes \nu_{w} d\mu^\Z(w)$$ 
defines a $\hat{T}$-invariant measure. The measure $\hat{\nu}$ is $\hat{T}$-ergodic if and only if $\nu$ is $\mu$-ergodic \cite[Chapter 2.6]{bq.book}.

\subsubsection{Stationary measures on the flag varieties}\label{subsub.stat.meas.flag}

Let $H$ be a real semisimple linear Lie group and $P$ a parabolic subgroup. According to a fundamental result of Furstenberg \cite{furstenberg.boundary.theory} (generalized to the current form by Guivarc'h--Raugi \cite{guivarch-raugi} and Goldsheid--Margulis \cite{goldsheid-margulis}), for any Zariski-dense probability measure $\mu$ on $H$, there exists a unique $\mu$-stationary probability measure on $H/P$. We shall refer to this measure as the Furstenberg measure and denote it by $\overline{\nu}_F$.

Recall that a stationary probability measure $\nu$ is said to be \textbf{$\mu$-proximal} if the limit measures $\nu_b$ are Dirac measures $\mu^{-\N^\ast}$-a.s. We will also say that the $H$-action on a space $Y$ is \textbf{$\mu$-proximal} if for every $\mu$-stationary and ergodic probability measure $\nu$ on $Y$ is $\mu$-proximal. For a Zariski-dense probability $\mu$, the Furstenberg measure (and hence the $H$ action on $H/P$) is $\mu$-proximal.

If $H$ acts $\mu$-proximally on a space $Y$, then every $\mu$-stationary probability measure $\nu$ induces a boundary map $w=(b,a) \mapsto \xi(w)=\xi(b)$ defined $\mu^{\Z}$-a.s.\ satisfying $\nu_w=\delta_{\xi(w)}$ and the equivariance property $b_{-1}\xi(T^{-1}w)=\xi(w)$. Conversely, a boundary map $\xi$ with the last equivariance property induces a $\mu$-proximal stationary probability measure. We will use the shorthand $P_{\mu}^{\erg}(Y)$ to denote the set of $\mu$-stationary and ergodic probability measures on $Y$.

\subsubsection{Limit measures on the fibre}\label{subsub.fibre.measures}
Let $Y_0$ and $Y=Y_0 \times F$ be $H$-spaces such that the projection $Y \to Y_0$ is $H$-equivariant. Let $\nu$ be a $\mu$-stationary probability measure on $Y$ such that its projection $\overline{\nu}$ on $Y_0$ (which is also automatically $\mu$-stationary) is $\mu$-proximal. Then, that for $\mu^{-\N^\ast}$-almost every $b$, the limit measure $\nu_{b}$ on $Y \simeq Y_0 \times F$ is of the form $\delta_{\xi(b)}\otimes \tilde{\nu}_{b}$, where $\xi: H^{-\N^\ast} \to Y_0$ is a measurable equivariant map (i.e.~ $\mu^{\Z}$-a.s. $\xi(ba_0)=a_0\xi(b)$) and $\tilde{\nu}_{w}=\tilde{\nu}_{b}$ is a probability measure on $F$.

\subsection{Measure classification for product systems with equivariant projections}\label{subsec.meas.class.sec2}

In the following result, we record, in a general setting, a description of stationary measures for actions on product spaces with equivariant projections on both factors. It is based on the Furstenberg decomposition of a stationary measure into its limit measures, i.e.~ $\nu=\int \nu_b d\mu^{-\N^\ast}(b)$.

\begin{proposition}\label{prop.decomposable.measure.class}
Let $H$ be a lcsc group, $Y_0$ and $F$ be lcsc $H$-spaces. Consider the $H$-action on $Y=Y_0 \times F$ for which both projections $Y \to Y_0$ and $Y \to F$ are $H$-equivariant. Let $\mu$ be a probability measure on $H$ such that the $H$-action on $Y_0$ or $F$ is $\mu$-proximal. Then, we have
$$
P_\mu^{\erg}(Y) \simeq P_\mu^{\erg}(Y_0)\times P_\mu^{\erg}(F).
$$
More precisely, the map
\begin{equation}\label{eq.bijection.map.decomposable}
\nu \mapsto (\overline{\nu},\nu^F)    
\end{equation}
is a bijection where the latters are, respectively, pushforwards of $\nu$ by the projections $Y \to Y_0$ and $Y \to F$.
\end{proposition}

The assumption of proximality induces a certain disjointness between two factors; it is clear that without such an assumption the conclusion fails (e.g.~ if $Y_0$ and $F$ have a common non-trivial factor).

\begin{proof}
Note that since the projections $Y \to Y_0$ and $Y \to F$ are both $H$-equivariant, the pushforward measures $\overline{\nu}$ and $\nu^F$ are both $\mu$-stationary. Moreover, it is clear that if $\nu$ is ergodic, then so are $\overline{\nu}$ and $\nu^F$. Let us show that the map $P_\mu^{\erg}(Y) \ni \nu \to P_\mu^{\erg}(Y_0)\times P_\mu^{\erg}(F)$ given by $\nu \mapsto (\overline{\nu},\nu^F)$ yields the desired bijection.
 
Without loss of generality, let us suppose that the $H$-action on $Y_0$ is $\mu$-proximal and  show that the above map is injective. Given $\nu \in P_\mu^{\erg}(Y)$, since the projections to each factor commute with the $H$-action, we have $\mu^{-\N^\ast}$-a.s.\ $\overline{\nu_b}=(\overline{\nu})_b$ and $(\nu_b)^F=(\nu^F)_b$. Moreover, since the $H$-action on $Y_0$ is $\mu$-proximal and $\overline{\nu}$ is ergodic, there exists a boundary map $\xi:B \to Y_0$ such that $(\overline{\nu})_b=\delta_{\xi(b)}$ for $\mu^{-\N^\ast}$-a.s.~ 
$b \in H^{-\N^\ast}$. Therefore, the probability measure $\nu_b$ is given by $\delta_{\xi(b)} \otimes (\nu^F)_b$ and hence by the Furstenberg decomposition, we can recover the measure $\nu$ as $\nu=\int \delta_{\xi(b)} \otimes (\nu^F)_b d\beta(b)$. This shows that the map $\nu \mapsto (\overline{\nu},\nu^F)$ is injective.

Surjectivity does not use the $\mu$-proximality assumption and follows from the fact that both projections $\overline{\nu}$ and $\nu^F$ are $\mu$-stationary and ergodic. Indeed, one readily checks that $\int (\overline{\nu})_b \otimes (\nu^F)_b d\beta(b)$ is a $\mu$-stationary and ergodic probability measure on $Y$.
\end{proof}

\section{$\SL_2(\R)$-Zariski closure: measure classification}\label{sec.meas.class}

We now begin the main part on classifying stationary measures on homogeneous bundles over flag varieties. Following the scheme exposed in the Introduction, in \S \ref{sub.base.and.cases}, we start by distinguishing Case 1 (Dirac base) and Case 2 (Furstenberg base) according to the classification in the base followed by a precise description of various possibilities that occur (see Figure \ref{figure.cases}) in Case 2. In the rest, we focus on Case 2. \S \ref{subsec.case.2.1} treats the trivial fibre case. In \S \ref{subsec.case.2.2.diag}, we treat the diagonal fibre action case and prove Theorem \ref{thm.measure.class.geod}. Finally, in \S  \ref{subsec.remaining.case}, we prove Theorem \ref{thm.irreducible.H.decompsable} and provide an example for Case 2.3.b.


\subsection{The setting, classification on the base and the cases}\label{sub.base.and.cases}
Let us start by recalling the notations from the introduction. Let $G$ be a semisimple real Lie group with finite centre and $Q<G$ a parabolic subgroup. Let $R_0 \unlhd Q$ be a normal algebraic subgroup and $R<Q$ be a closed subgroup containing $R_0$ such that $S:=Q/R_0$ is semisimple with finite centre and without compact factors and $\Lambda:=R/R_0$ is a discrete subgroup of $S$. We denote by $X$ the quotient space $G/R$. As explained in the introduction, the space $X$ has a natural $G$-equivariant projection to $G/Q$ endowing it with a fibre-bundle structure over the flag variety $G/Q$ with fibres given by copies of $S/\Lambda$.

A convenient way to find such subgroup $R_0 \unlhd Q$ as above is by considering the  refined Langlands decomposition $Q=S_Q E_QA_QN_Q$ of $Q$ (see e.g.~ \cite[VII,7]{knapp.book}), where $S_Q$ is a semisimple subgroup of $Q$ without compact factors, $E_Q$ is a compact subgroup commuting to $S_Q$, $A_Q$ is a maximal $\R$-split diagonalizable subgroup commuting with $S_QE_Q$, and $N_Q$ is the unipotent radical of $Q$. One can then take the normal subgroup $R_0$ of $Q$ to be of the form $S'_QE_Q A_Q N_Q$ where $S'_Q$ is a simple factor of $S_Q$ or the trivial group.

Let $\mu$ be a probability measure on $G$ with finite first moment, $\Gamma_\mu$ the closed semigroup generated by the support of $\mu$ and suppose that the Zariski-closure $\overline{\Gamma}_\mu^Z$ of $\Gamma_\mu$ is a copy of either $\PGL_2(\R)$ or $\SL_2(\R)$ in $G$. We will denote this Zariski closure by $H$. On our way to establishing a description of $\mu$-stationary probability measures $\nu$ on $X$, we remark that a measure $\nu$ on $X$ determines, and in turn is determined by, its projection on $G/Q$ via $\pi: X \to G/Q$ and the fibre measures of this projection. Therefore, we proceed by first discussing the possible $\mu$-stationary measures on the base $G/Q$.

\subsubsection{Stationary measures on the base}\label{subsub.stat.meas.base}
Since the projection $X \to G/Q$ is $G$-equivariant, any $\mu$-stationary probability measure $\nu$ on $X$ projects down to $\mu$-stationary probability measure $\overline{\nu}$ on the base $G/Q$. The description of stationary measures on the base is handled by Guivarc'h--Raugi \cite{guivarch-raugi} and by Benoist--Quint \cite{BQ.compositio}. We have the following.

\begin{lemma}\label{lemma.base.cases}
There exists a bijection between $\mu$-stationary and ergodic probability measures on $G/Q$ and compact $H$-orbits on $G/Q$.
\end{lemma}

\begin{proof}
It is clear that any compact $H$-orbit carries a $\mu$-stationary and ergodic probability measure. Conversely, let $\nu$ be a $\mu$-stationary and ergodic probability measure. By Chacon--Ornstein ergodic theorem, there exists $x \in G/Q$ such that $\frac{1}{N} \sum_{k=1}^N \mu^{\ast k} \ast \delta_x \to \nu$ as $N \to \infty$. So, in particular, $\nu$ is supported in the compact $\overline{\Gamma_\mu x}$. Since $G/Q$ is a flag variety, the orbit $Hx$ is locally closed (see e.g.~\cite[Theorem 3.1.1]{zimmer.book}). Hence the compact $\overline{\Gamma_\mu x}$ is contained in $Hx$. Moreover, since $Hx$ is locally compact, up to conjugating $Q$, it is ($H$-equivariantly) homeomorphic to $H/Q_H$ where $Q_H=H \cap Q$. It now follows from \cite[Proposition 5.5]{BQ.compositio} that $Q_H$ is cocompact in $H$ and $\nu$ is the unique $\mu$-stationary and ergodic probability measure supported in $H/Q_H \simeq Hx$. This concludes the proof.
\end{proof}

It follows from this result that there are two types of $\mu$-stationary and ergodic probability measures on the base $G/Q$. The first type, which we will refer to as \textbf{Case 1}, is Dirac measures. This happens if and only if $H$ is contained in a conjugate of $Q$. The second type (\textbf{Case 2}) is the Furstenberg measure supported in a copy of $H/P$ in $G/Q$ where $P$ is a parabolic subgroup of $H$. This happens if and only if $H$ intersects a conjugate of $Q$ in a parabolic subgroup. Note that both types of stationary measures can be simultaneously present in a $G$-homogeneous bundle $X$.

Our study will primarily concern the analysis of stationary probability measures falling in Case 2. Indeed, as we now discuss, Case 1 is handled precisely by the seminal works of Benoist--Quint \cite{BQ1,BQ2} and  Eskin--Lindenstrauss \cite{eskin-lindenstrauss.long}.

\subsubsection{Case 1: Dirac base}\label{subsub.dirac.base}\label{subsub.dirac.base}

Let us observe that the results of Benoist--Quint \cite{BQ1,BQ2} and Eskin--Lindenstrauss \cite{eskin-lindenstrauss.long} imply the following.

\begin{proposition}
Keep the setting above and let $\nu$ be a $\mu$-stationary and ergodic probability on $X$ whose projection onto $G/Q$ is a Dirac measure. Then the fibre measure $\nu^F$ of $\nu$ on $(Q/R_0)/(R/R_0)\simeq S/\Lambda$ is homogeneous.
\end{proposition}

This result follows from a direct application of \cite[Theorem 1.3]{eskin-lindenstrauss.long} which extends the main measure classification results of Benoist--Quint \cite{BQ2} with regards to the moment assumption and the fact that the group $\Lambda$ is only required to be discrete.

\begin{proof}
By assumption, the projection $\overline{\nu}$ of $\nu$ is a Dirac $\delta_{gQ}$ on $G/Q$. Replacing if necessary $Q$ by a conjugate, we can suppose $H<Q$ and that $g=\id$. The fibre of the map $X \to G/Q$ above $\id Q$ identifies $Q$-equivariantly with $Q/R \simeq (Q/R_0)/(R/R_0)$. Identifying $\mu$ with its image in $HR_0/R_0<Q/R_0$, the measure $\nu^F$  is  $\mu$-stationary and ergodic. Therefore, it follows from \cite[Theorem 1.3]{eskin-lindenstrauss.long} that $\nu^F$ is a homogeneous measure.
\end{proof}

\subsubsection{Case 2: Furstenberg base}\label{subsub.furstenberg.base}

The rest of this section is devoted to the analysis of the remaining case, i.e.~ the description of a  $\mu$-stationary and ergodic probability on $X$ whose projection to $G/Q$ is non-atomic (and which is consequently the Furstenberg measure on a copy of $H/P$ in $G/Q$, where $P$ is a parabolic subgroup of $H$, see Lemma \ref{lemma.base.cases}). In this case $H$ intersects a conjugate $gQg^{-1}$ of $Q$ in a parabolic subgroup. By conjugating $Q$ if necessary, we can and will suppose that $g=\id$, i.e.~ $P=Q_H:=H \cap Q$ and $\nu$ lives in $\pi^{-1}(H/P)$. As before, the analysis of stationary measure will vary depending on the relative position of $H$ with respect to the parabolic group $Q$ within the ambient group $G$. We will distinguish three cases that will be dealt with separately in the following subsections.\\

\noindent \textbullet ${}$ \textit{Case 2.1}: Trivial fibre action. This is a trivial case that occurs when the parabolic subgroup $Q_H^o$ of $H$ is contained in $R_0$.\\[4pt] 
\textbullet ${}$ \textit{Case 2.2}: Diagonal fibre action. This is the case when $Q_H^o \cap R_0$ is a proper non-trivial subgroup of $Q_H^o$. As we shall see, this positioning gives rise to a situation where the $S$-valued cocycle describing the fibre action has values in a diagonal subgroup of $S$.\\[4pt]
\textbullet ${}$ \textit{Case 2.3}: This is the remaining case, i.e.~ the case where $Q_H^o \cap R_0$ is trivial. Interestingly, the analysis in this case depends on further properties $H$ with respect to $G$ that we will explain. Accordingly, our analysis will involve two subcases (\textit{2.3.a} and \textit{2.3.b}). In this paper, we will not be able to give the full description of stationary measures in the last of these two subcases.

\subsection{Case 2.1: Trivial fibre action}\label{subsec.case.2.1}

We express the description in this simple case in the following result.

\begin{proposition}\label{prop.trivial.fibre.measure.class}
Let the space $X$ and groups $G,Q,R_0,R,S\simeq Q/R_0,\Lambda\simeq R/R_0$ and $H$ be as defined before. Suppose that $Q_H^o$ is contained in $R_0$. Then, there exists a standard trivialization $X \simeq G/Q \times S/\Lambda$ such that any $\mu$-stationary and ergodic probability measure $\nu$ on $X_\mathcal{C}$ can be written as $\overline{\nu}_F \otimes \delta_{q\Lambda}$ a product of the Furstenberg measure $\overline{\nu}_F$ with a Dirac measure $\delta_{q\Lambda}$ for some $q \in Q$. 
\end{proposition}

\begin{proof}
Start by noting that since $Q_H$ is a (Zariski) connected algebraic group, $Q_H^o<R_0$ implies that $Q_H<R_0$. Now fix any standard trivialization $X_\mathcal{C} \simeq H/Q_H \times S/\Lambda$, let $\beta$ be the associated cocycle. The hypothesis $Q_H<R_0$ then entails that the associated morphism $\rho_\beta:Q_H \to S$ has trivial image. In particular $\rho_\beta$ extends trivially to a morphism $H \to S$ and hence by Proposition \ref{prop.decomposable}, the $H$-action on $X_\mathcal{C}$ is decomposable.
Therefore there exists a standard trivialization $X_\mathcal{C} \simeq H/Q_H \times S/\Lambda$ for which the associated cocycle is morphism-type with trivial morphism. The result follows.
\end{proof}

Here are two examples where the trivial fibre action situation arises. 

\begin{example}\label{ex.case.2.1}
1. (Trivial example) Let $G=\SL_4(\R)$, $Q$ be the minimal parabolic subgroup preserving the standard full flag in $\R^4$. Let $H$ be the copy of $\SL_2(\R)$ on the top-left corner, i.e.~ acting on the plane generated by the standard basis elements $e_1$ and $e_2$. In this case, $R$ is necessarily equal to $R_0$ which is $Q$ itself.

2. Let $G=\SL_4(\R)$, $Q$ be the parabolic subgroup stabilizing the plane generated by the first two vectors $e_1,e_2$ of the standard basis of $\R^4$, $H$ be the reducible representation given by the sum of the standard representation of $\SL_2(\R)$ on the planes generated by the basis vectors $e_1,e_4$ and $e_2,e_3$;
$$
Q= \left \{ \begin{pmatrix}
\ast & \ast & \ast & \ast \\
\ast & \ast & \ast & \ast \\
0 & 0 & \ast & \ast \\
0 & 0 & \ast & \ast \\
\end{pmatrix}
\right \}, \quad  \quad H = \left \{ \begin{pmatrix}
a & 0 & 0 & b \\
0 & a & b & 0 \\
0 & c & d & 0 \\
c & 0 & 0 & d \\
\end{pmatrix} | \begin{pmatrix}
a & b \\
c & d 
\end{pmatrix} \in \SL_2(\R)
\right \}.
$$
We can take $R$ to be the group generated by $R_0$ which is the solvable radical of $Q$ and $\SL_2(\Z) \times \SL_2(\Z)$ acting in the standard way on the planes generated by $e_1,e_2$ and $e_3,e_4$. 
\end{example}

\begin{remark}
    Case 1 and Case 2.1 work more generally and we only need the assumption that $H$ is a real semisimple linear Lie group.
\end{remark}

\subsection{Case 2.2: Diagonal fibre action}\label{subsec.case.2.2.diag}

The main goal of this part is to prove Theorem \ref{thm.measure.class.geod}. We start by discussing an example to which this result applies.



\begin{example}\label{ex.reducible}
We start by recalling Example \ref{ex.ss}. Let $X$ be the space of $2$-lattices in $\R^3=V$ up to homotheties of $V$. The space $X$ with its natural topology admits a continuous transitive action of $G=\SL_3(\R)$. Let $y_0<V$ be a copy of $\R^2$ generated by $e_1,e_2$, where $e_i$'s denote the standard base elements of $V$ and $x_0$ be the class of $\Z^2$ in $y_0$, $R$ its stabilizer in $G$ and $Q$ the parabolic subgroup of $G$ stabilizing $y_0$. Note that the connected component $R_0$ of $R$ is the solvable radical of $Q$ and we have $R<Q$ so that the space $X$ is a $G$-homogeneous bundle over $G/Q$ with fibres $Q/R \simeq \PGL_2(\R)/\PGL_2(\Z)$. Let $\pi:X \to G/Q$ denote the natural projection associating to a class of $2$-lattice the $2$-plane that it generates.

Let $H$ be a copy of $\SL_2(\R)$ given by the classes of matrices of the form $ \begin{pmatrix}
1 & 0 & 0 \\
0 & a & b \\
0 & c & d
\end{pmatrix}$, where $ \begin{pmatrix}
a & b \\
c & d
\end{pmatrix} \in \SL_2(\R)$. 
This configuration falls into Case 2.2: indeed, $Q_H=H \cap Q$ is a parabolic subgroup of $H$ and $Q_H^\circ \cap R_0$ is the unipotent radical of $Q_H^\circ$.

One can also see in explicitly how the one-dimensional split subgroup of $S$, for which Theorem \ref{thm.measure.class.geod} proves invariance, appears: let $\mathcal{C}$ be the $H$-invariant circle of 2-planes given by $\langle e_1, t e_2+se_3 \rangle$ for $t,s \in \R$ and $X_{\mathcal{C}}$ be the bundle over $\calC$ given by the closed subset of $X$ given by $2$-lattices contained in subspaces belonging to $\mathcal{C}$. We can choose an explicit Borel section $s:H/Q_H \to H$ as follows to obtain a standard trivialization of $X_\mathcal{C}$: given a vector space $y=\langle e_1, t e_2+se_3 \rangle \in \mathcal{C} \simeq H/Q_H$, we can associate the class of the matrix $R_{\theta(y)}:=\begin{pmatrix}
\cos \theta(y) & -\sin \theta(y) \\
\sin \theta(y) & \cos \theta(y)
\end{pmatrix}
$ with $\theta(y) \in [0,\pi)$ chosen so that $R_{\theta(y)}$ seen in $H$ sends $y_0$ on $y$. 
The resulting trivialization writes as 
\begin{equation*}
\begin{aligned}    
X_\mathcal{C} &\simeq H/Q_H \times F\\
x &\mapsto (\pi(x), R_{\theta(\pi(x))}^{-1}(x)).
\end{aligned}
\end{equation*}
A straightforward calculation shows that 
the $S \simeq \PGL_2(\R)$-valued cocycle $H \times H/Q_H \to S$ given by this trivialization takes values in the full diagonal subgroup $D^{\pm}<\PGL_2(\R)$ and coincides with the Iwasawa cocycle of $H$ up to a sign, which will be defined in the following section. It follows then from Theorem \ref{thm.measure.class.geod}, Remark \ref{rk.conversely.to.diagonal.thm} and uniqueness of $\mu$-stationary measure on $H/Q_H$ that there is a bijection between diagonal-flow (or an index-two extension of it) invariant probability measures on $\PGL_2(\R)/\PGL_2(\Z)$ and $\mu$-stationary probability measures on $X_\mathcal{C}$. The difference between diagonal invariance and the index-two extension is a minor one related to the sign group. This is discussed further below in \S \ref{subsub.iwasawa.sign}.
Note finally that Case 1 also appears within this same example, namely the singleton corresponding to the two-plane generated by $\{e_2,e_3\}$ is $H$-invariant.
\end{example}

\bigskip

The rest of this Subsection \ref{subsec.case.2.2.diag} is devoted to the proof of Theorem \ref{thm.measure.class.geod}.


\subsubsection{Iwasawa cocycle and representation theory}\label{subsub.iwasawa}

Let $H$ denote either the group $\SL_2(\R)$ or $\PGL_2(\R)$. Let $K$ be a maximal compact subgroup of $H$ and $P$ be a minimal parabolic subgroup so that we have the decomposition $H=K^oP $. Let $D$ be the maximal connected diagonal subgroup of $P$, $N$ the unipotent radical of $P$ and $M=K \cap P$. Let $H/P$ be the flag variety of $H$. Given $h \in H$ and $\xi=kP \in H/P$, we denote by $\sigma(h,\xi)$ the unique element of $D$ such that
\begin{equation}\label{eq.iwasawa.characterizing}
hk \in K \sigma(h,\xi) N.    
\end{equation}
This map $\sigma: H \times H/P \to D$ defines a continuous cocycle (see \cite[Lemma 8.2]{bq.book}), called the \textit{Iwasawa cocycle}. The morphism $\rho_\sigma$ associated with the Iwasawa cocycle is simply the projection $P \to D \simeq P/MN$.

An alternative way, more in the spirit of Section \ref{sec.prelim}, to construct the Iwasawa cocycle is as follows. Consider $H$ as a fibre bundle over $H/P$ and let $s$ be a section $s:H/P \to H$ given by the Iwasawa decomposition, namely $s(kP)\in kM$ for $k \in K^\circ$.  We then get a trivialization $H \simeq H/P \times P$ (see \eqref{eq.trivialization}) and an associated cocycle $\tilde{\sigma}:H \times H/P \to P$ (see \eqref{eq.construct.cocycle}). It is not hard to verify that the cocycle obtained by composing $\tilde{\sigma}$ with the projection $P \to P/MN \simeq D$ satisfies the characterizing property \eqref{eq.iwasawa.characterizing} of the Iwasawa cocycle. Regarding the section $s$, for $\PGL_2(\R)$ case, we can define it canonically to have values in $K^\circ$. For $\SL_2(\R)$ case, we need to make a choice in $kM$ so that $s(kP)$ is a Borel section. Even though the cocycle $\tilde{\sigma}$ depends on $s$, by \eqref{eq.construct.cocycle}, since the ambiguity $M$ is in the centre, we know that Iwasawa cocycle does not depend on the choice of the value of $s(kP)$ in $kM$.


In the course of our proofs, sometimes it will be more convenient to switch to the additive notation for cocycles. Let $\mathfrak{d}$ be the Lie algebra of $D$. For a $D$-valued cocycle $\alpha$, we will denote by $\overline{\alpha}$, the $\mathfrak{d}$-valued cocycle obtained by composing $\alpha$ with the logarithm map $D \to \mathfrak{d}$.


Given an algebraic irreducible representation $\rho: H \to \GL(V)$ of $H$ in a finite dimensional real vector space $V$, for every character $\chi$ of $D$, the associated weight space is
$V^\chi=\{v \in V : \rho(a) v= \chi(a)v \; \; \text{for every} \; a \in D \}$. The set of characters $\chi$ for which $V^\chi \neq \{0\}$ is called the set of (restricted) weights of $(V,\rho)$ and denoted $\Sigma(\rho)$. For a character $\chi$ of $D$, we denote by $\overline{\chi}$ the corresponding additive character on $\mathfrak{d}$. The set $\Sigma(\rho)$ is endowed with an order: $\overline{\chi}_1 \geq \overline{\chi}_2$ if and only if $\overline{\chi}_1 - \overline{\chi}_2$ is a sum of positive roots of $H$ in $\mathfrak{d}$. $\Sigma(\rho)$ has a largest element $\chi$, called the highest weight of $\rho$. The corresponding eigenspace is the subspace $V^N$ of $N$-fixed vectors. Since $H$ is $\R$-split, this is a line in $V$. For an element $\eta=gP$ in the flag variety $H/P$, we denote by $V_\eta$ the line $gV^N$ in $V$ constructing a map $H/P \to \mathbb{P}(V)$.

The following lemma will be important for our considerations; it will allow us to control the Iwasawa cocycle.

\begin{lemma}\cite[Lemma 6.33]{bq.book}\label{lemma.iwasawa.norm}
Let $(V,\rho)$ be an algebraic irreducible representation of $H$ with the highest weight $\chi$. Then, there exists a $K$-invariant Euclidean norm $\|.\|$ on $V$ such that for every element $a \in D$, $\rho(a)$ is a symmetric endomorphism of $V$. Moreover, for every $\eta \in H/P$, non-zero $v \in V_\eta$ and $h \in H$, we have
$$
\overline{\chi}(\overline{\sigma}(h,\eta))=\log  \frac{\|\rho(h)v\|}{\|v\|}.
$$
\end{lemma}

\subsubsection{Iwasawa cocycle, sign group and standard trivialization}\label{subsub.iwasawa.sign}

The goal of this part is to obtain a lemma (Lemma \ref{lemma.its.iwasawa} below) which, for a standard trivialization, expresses the action of $H$ on the fibres of the subbundle $X_{H/Q_H}$ of $X \to G/Q$ with the Iwasawa cocycle of the group $H$, up to a sign. 
Recall that in Case 2.2, $R_0 \cap Q_H^\circ$ is a non-trivial proper subgroup of $Q_H^\circ$. Since $R_0$ is normal in $Q$, $R_0 \cap Q_H^\circ$ is also normal in $Q_H^\circ$. It follows  that this intersection is the unipotent radical of $Q_H^\circ$.
Therefore the projection of $Q_H^\circ$ to $S$ given by $Q_H^\circ/(Q_H^\circ\cap R_0)$ is a connected split torus.
We will denote by $D$ the image of $Q_H^\circ$ in $S$ obtained by projection. Let $D^{\pm}$ be the algebraic $\R$-split torus containing $D$. Then $Q_H/(Q_H\cap R_0)$, the image of $Q_H$ in $S$, is contained in $D^{\pm}$. The group $D^{\pm}\simeq \R^*$ is isomorphic to $D\times(\Z/2\Z)\simeq \R_{>0}\times \{\pm 1\}$ in $S$. 
In order to treat the sign problem of the cocycle in $D^\pm$, we need to go to the two-fold cover space $K$ of $H/Q_H$ to recover the information of the sign. Here we need to distinguish two cases in a similar way for both $H \simeq \SL_2(\R)$ or $\PGL_2(\R)$.

For $H \simeq \SL_2(\R)$ case: Let $V=\R^2$. A convex cone in $V$ is called proper if it does not contain a line. From Iwasawa cocycle, or just from the action of $\SL_2(\R)$ on $\bbS\subset V$, we have a group action of $\SL_2(\R)$ on $K\simeq \bbS$. Guivarc'h and Le Page \cite[Proposition 2.14]{GLP} proved that if $\Gamma_\mu$ preserves a closed proper convex cone in $V$ then there exist two $\mu$-stationary and ergodic measures $\nu_1$ and $\nu_2$ on the circle $K$. The supports of these two measures are just the inverses of each other, and we denote them by $\Lambda_1$ and $-\Lambda_1$, respectively. Otherwise, there exists a unique $\mu$-stationary measure on $K$. We now distinguish two cases depending on the action of $\Gamma_\mu$ on $K$. 
\begin{itemize}
    \item Case 2.2.a: $\Gamma_\mu$ preserves a closed proper convex cone in $V$. In this case we take a section $s:K/M\to K$ such that $s$ takes values in a half circle containing $\Lambda_1$.
    \item Case 2.2.b: Otherwise. We just take a section $s:K/M\to K$. There is no better choice in this case.
\end{itemize}

For $H\simeq \PGL_2(\R)$ case: The maximal compact subgroup $K$ has two connected components and each component is isomorphic to $H/Q_H$. In this case, we take the section $s:K/M\to K$ in the connected component of $K$.
\begin{itemize}
    \item Case 2.2.a: If $\Gamma_\mu$ is inside the connected component $\PGL_2(\R)^\circ$. 
    \item Case 2.2.b: Otherwise. In this case, we have a unique $\mu$-stationary measure on $K$, which has weight $1/2$ on each connected component.
\end{itemize}

We mention that unlike other main cases (those appearing in Figure \ref{figure.cases}), Cases 2.2.a or b depend on $\Gamma_\mu$ rather than the Zariski closure $H$ itself.

For $H$ equal to either $\SL_2(\R)$ or $\PGL_2(\R)$, from now on we distinguish Case 2.2.a and Case 2.2.b, and choose a Borel section $s:H/Q_H\simeq K/M\to K<H$ as specified above. We define a sign function on $K$ by 
\[\sg(k):=k^{-1}s(kM)\in M\simeq \Z/2\Z. 
\]
We define a sign cocycle with respect to the section $s$ for $g\in H$ and $\eta\in K/M$ by
\[\sg(g,\eta):=\sg(k)\sg(k_g)=k^{-1}s(kM)k_g^{-1}s(k_gM), \]
where $k$ is a preimage of $\eta$ in $K$ and $k_g$ is the $K$-part of $gk\in k_g\sigma(g,k)N$ in the Iwasawa decomposition. The value of $\sg$ does not depend on the choice of preimage $k$. 

In Case 2.2.b, with this sign function, we can recover the sign in $D^\pm$ of the cocycle $\alpha$. Recall the quotient map $Q_H$ to $S$, whose image is $Q_H/(Q_H\cap R_0)<D^{\pm}$. If $Q_H/(Q_H\cap R_0)=D$, then there is no ambiguity about the sign. In the following, in order to simplify the notation, we suppose that we are in the case where $Q_H/(Q_H\cap R_0)=D^\pm$. The proof of the case $Q_H/(Q_H\cap R_0)=D$ is simpler; the sign cocycle disappears, or equivalently, it is constant with value identity.

\begin{lemma}\label{lem:sign}\label{lemma.its.iwasawa}
Under the above choice of the section $s$, for $g\in H$ and $\eta\in H/Q_H$, as an element in $D^\pm$, we have
\[\alpha(g,\eta)=(\sigma(g,\eta),\sg(g,\eta)). \]
In particular, for Case 2.2.a, when $\eta$ is in the support of the Furstenberg measure and $g\in\Gamma_\mu$, the cocycle $\alpha$ coincides with the Iwasawa cocycle.
\end{lemma}
\begin{proof}
By definition of the Borel section $s$ and cocycle $\alpha$, 
\[\alpha(g,\eta)=s(g\eta)^{-1}gs(\eta)R_0\in S=Q/R_0. \]
The Iwasawa cocycle is defined by
\[\sigma(g,\eta)=k_g^{-1}gkN. \]
Recall that $k$ is a preimage of $\eta$ in $K$ and $k_g$ is the $K$-part of $gk\in k_g\sigma(g,k)N$ in the Iwasawa decomposition.
The difference of the sign comes from the product of the differences of the signs of $k,s(\eta)$ and $k_g,s(g\eta)$. By definition of the sign cocycle, we obtain the formula for $\alpha(g,\eta)$.

Regarding the second statement, in the case of $H \simeq \PGL_2(\R)$, it is a consequence of positive determinant. For $H \simeq \SL_2(\R)$, since the action of $\Gamma_\mu$ preserves $\Lambda_1$ inside $K$, if we take $k$ in $\Lambda_1$, then $k_g$ is still in $\Lambda_1$ for $g\in \Gamma_\mu$. In this case we obtain that the sign cocycle $\sg$ is identically equal to $\id$ for the $\Gamma_\mu$-action, whence the claim.
\end{proof}

It follows from this lemma that our choice of the section $s:G/Q \to G$ implies that the cocycle $\alpha$ for the associated standard trivialization, projected on $D$, is equal to the Iwasawa cocycle $H \times H/Q_H  \to D$. 
In the rest of this part (Case 2.2), we will work with this choice of coordinates on $X$ (i.e.~ trivialization induced by the section $s$). We will identify the space $Q/R$ with the quotient $S/\Lambda$ where $S\simeq Q/R_0$ and $\Lambda$ is the lattice $R/R_0$. To alleviate the notation, sometimes we will write $F=S/\Lambda$.

\subsubsection{Limit measures on the fibre}

Let $\mu$ be a Zariski-dense probability measure on $H$ and $\nu$ be a $\mu$-stationary probability measure on $H/Q_H \times_\alpha F$. Recall from \S \ref{subsub.stat.meas.flag} that $\mu$ admits a unique stationary probability measure $\overline{\nu}_F$ on $H/Q_H$ which is also $\mu$-proximal (the Furstenberg measure). It follows (see \S \ref{subsub.fibre.measures}) that $\mu^{-\N^\ast}$-almost every $w^-$, the measure $\nu_{w^-}$ on $H/Q_H \times_\alpha F$ is of the form $\delta_{\xi(w^-)}\otimes \tilde{\nu}_{w^-}$, where $\xi: H^{-\N^\ast} \to H/Q_H$ is a measurable equivariant map (i.e.~ $\xi(ba_0)=a_0\xi(b)$ for $\mu^\Z$-a.e.~ $w$) and $\tilde{\nu}_{w^-}=\tilde{\nu}_{b}$ is a probability measure on $F$. In view of the equivariance property of $\nu_{w^-}$ and the fact that the action on the $F$-coordinate is given by the cocycle $\alpha$ over $H/Q_H$, the measure $\tilde{\nu}_{b}$ satisfies the following equivariance formula for $\mu^{-\N^\ast}$-a.e.~ $b \in H^{-\N^\ast}$ and $\mu$-a.e.~ $a_0 \in H$,
\begin{equation}\label{eq.equiv.cocycle}
\tilde{\nu}_{ba_0}=\alpha(a_0,\xi(b))\tilde{\nu}_b.    
\end{equation}
In the sequel, to simplify the notation, we also use the notation $\nu_{w^-}$ (or $\nu_w$) to denote the fibre measure $\tilde{\nu}_{w^-}$. This should not cause confusion. 

We start with a first claim which will allow us to focus attention on a single generic fibre measure $\nu_{w^-}$.

\bigskip

\noindent \textbf{Claim 0:} To prove Theorem \ref{thm.measure.class.geod}, it suffices to show that for $\mu^{-\N^\ast}$-a.e. $w^-$, the measure $\nu_{w^-}$ on $S/\Lambda$ is $D$-invariant.

\bigskip


\textit{Proof of Claim 0:} 
If we are in Case 2.2.a, by Lemma \ref{lemma.its.iwasawa}, the cocycle $\alpha$ actually takes values in $D$. From the equivariance formula \eqref{eq.equiv.cocycle} for $\nu_{w^-}$, 
it follows that the map $w^- \mapsto \nu_{w^-}$ is invariant under the inverse of the shift $T$ and hence is almost surely constant, by ergodicity of the map $T^{-1}$. 

For Case 2.2.b, we need to consider an extension by $\Z/2\Z=\{\pm 1 \} \simeq M$. We define $T^\sg$ on $H^{\Z}\times \Z/2\Z$ by 
\[T^\sg(w,j)=(Tw, \sg(w_0,\xi(w^-))j), \]
where $w\in H^\Z$ and $j\in \Z/2\Z$. For $\mu^\Z$-a.e. $w\in H^\Z$, we define
\[\nu_{w,1}=\nu_{w^-},\ \nu_{w,-1}=(-1)_*\nu_{w^-}, \]
where $(-1)_*$ is understood as the action of $-id\in D^\pm$. 
Then the formula \eqref{eq.equiv.cocycle} and Lemma \ref{lem:sign} imply
\[ {\nu}_{T^\sg(w,j)}=\sigma(w_0,\xi(w^-)){\nu}_{w,j}.  \]
Now since the Iwasawa cocycle $\sigma$ takes values in $D$ and $\nu_{\omega,j}$'s are $D$-invariant by running the same argument as in Case 2.2.a, we see that it is sufficient to prove that the measure $\beta^\sg:=\mu^\Z\otimes ((\delta_1+\delta_{-1})/2)$ is $T^\sg$-ergodic. We now proceed to prove this.

We consider the $\mu^\Z$-a.e.~ defined map $p$ from $H^\Z\times \Z/2\Z$ to $ H^\Z\times K$, by letting 
\[ p(w,j)=(w,k_{w^-} ),\text{ where } \sg(k_{w^-})=j, k_{w^-}Q_H=\xi(w^-). \]
Let $\tilde T^\sg(w,k)=(Tw, w_0k)$. 
The pushforward of the measure $\beta^\sg$ yields the measure
\[\tilde\beta^\sg:=\int_{H^\Z} \delta_w\otimes ((\delta_{k_{w^-}}+\delta_{-k_{w^-}})/2)\ d\mu^\Z(w) \]
on $H^\Z\times K$. Then $p$ is a semiconjugacy from $(H^\Z\times\Z/2\Z,T^\sg,\beta^\sg)$ to $(H^\Z\times K,\tilde T^\sg,\tilde \beta^\sg)$.
The fiber measure $(\delta_{k_{w^-}}+\delta_{-k_{w^-}})/2 $ is actually the measure $(\nu_K)_w$ for the unique $\mu$-stationary measure $\nu_K$ on $K$ and $(\nu_K)_w$ is the limit of $b_{-1}\cdots b_{-n}\nu_K$ for $\mu^Z$-a.e.~ $w$. (This measure $(\nu_K)_w$ is a lift of the Dirac mass $\delta_{\xi(w^-)}$ on $H/Q_H$. Since $\nu_K$ is unique, we can verify that the limiting measure has equal mass on two preimages). By \cite[Section 2.6]{bq.book}, since $\nu_K$ is $\mu$-ergodic, we know that $\tilde\beta^\sg$ is $\tilde T^\sg$ ergodic. Then from the semiconjugacy $p$, we obtain that $\beta^\sg$ is $T^\sg$ ergodic. The proof is complete.
\qed

\begin{remark}[$D^{\pm}$-invariance in Case 2.2.b]
In Case 2.2.b, the argument above implies that $\nu_{w,1}=\nu_{w,-1}$ for $\mu^\Z$-a.e.~ $w$. So the fiber measure $\nu^F$ is indeed $D^\pm$-invariant. We will also see later in the equidistribution part that the limiting measure will be $D^\pm$-invariant.
\end{remark}

\subsubsection{Dynamically defined norms}

To obtain the required $D$-invariance for a typical limit measure $\nu_{w^-}$ on the fibre, using the equivariance formula \eqref{eq.equiv.cocycle}, we will be passing to a limit of cocycle differences of type $\alpha(a_m'\ldots a_0', \xi(b))\alpha(a_n \ldots a_0, \xi(b))^{-1}$ for various sequences $b$ and $a$ as well as carefully chosen times $m,n \in \N$. The choice of times and sequences will be made so that the sequences land in some nice compact subset of the shift space and, simultaneously, the cocycle differences are controlled. An important tool for this purpose will be the dynamically defined norms given by the next result.

We fix an irreducible algebraic representation $V$ of $H$, where $V$ is a finite-dimensional real vector space. Endow it with a $K$-invariant Euclidean structure and let $\|\cdot\|$ be the standard Euclidean norm on $V$. 
Here and below, we will also use the shorthand $a^n$ to denote the finite product $a_{n-1} \ldots a_0$ of the corresponding sequence $(a_0,\ldots, a_{n-1}) \in H^n$. We have

\begin{proposition}\label{prop.dynnorm} \cite[Proposition 2.3]{eskin-lindenstraus.short} 
There exists a measurable map $w \mapsto \|.\|_w$ from $H^\Z$ into the space of Euclidean norms on $V$ and a $T$-invariant full measure subset $\Psi$ of $H^{\Z}$ such that for every $w=(b,a) \in \Psi$ and $n\in\N$, letting
\[\lambda_1(w,n):=\log\frac{\|a^nv_b \|_{T^n w}}{\|v_b \|_w}, \]
there exists $\kappa>1$ such that
\[\lambda_1(w,n)\in [1/\kappa,\kappa]n.  \]
In particular, due to cocycle property, for $w\in\Psi$ and $m>n$ in $\N$,
\begin{equation}\label{eq.lip.dynnorm}
\lambda_1(w,m)-\lambda_1(w,n)\in[1/\kappa,\kappa](m-n) .
\end{equation} 
\end{proposition}

We note at this point that the finite first moment assumption in Theorem \ref{thm.measure.class.geod} is required in the proof of the previous proposition in \cite{eskin-lindenstraus.short}.

This norm $\|.\|_w$ is called \textit{dynamical defined norm}. It is chosen with respect to the dynamics such that Proposition \ref{prop.dynnorm} holds.
Due to measurability of $w \mapsto \|.\|_w$, we can always compare the dynamically defined norms and the original norm on a large measure subset of $H^\Z$.
\begin{lemma}\cite[Lemma 2.7]{eskin-lindenstraus.short} \label{lemma.comparison.dyn.norm}
For every $\delta>0$, there exists a compact subset $K(\delta)$ of $\Psi$ with $\mu^{\Z}(K(\delta))>1-\delta/10$ and a constant $C(\delta)>0$ such that for $v\in V$ and $w\in K(\delta)$
\[ 1/C(\delta)\leq \frac{\|v\|_w}{\|v\|}\leq C(\delta). \]
\end{lemma}

We denote by $\chi$ the highest weight of the representation from Lemma \ref{lemma.iwasawa.norm}. Combining Lemmas \ref{lemma.iwasawa.norm} and \ref{lemma.comparison.dyn.norm}, and Proposition \ref{prop.dynnorm}, we deduce the following

\begin{corollary}\label{corol.alpha.close.to.dynamical.cocycle}
For every $\delta>0$, there exists a compact subset $K(\delta)$ of $H^\Z$ with $\mu^{\Z}(K(\delta))>1-\delta/10$ and a constant $C(\delta)>0$ such that for every $w \in \Psi$ and $n \in \N$ such that $w$ and $T^n w$ are both in $ K(\delta)$, we have
\begin{equation}\label{eq.comparison.dyn.norm}
    |\overline{\chi}(\overline{\alpha}(a_{n-1}\ldots a_0,\xi(b)))-\lambda_1(w,n)|\leq C(\delta).
\end{equation}
\end{corollary}

\begin{proof}
Recall from \S \ref{subsub.stat.meas.flag} that given a Zariski-dense probability measure $\mu$ on $H$, we have a map $\xi: H^\Z \to H/P$ defined for $\mu^\Z$-a.e.~ $w=(b,a)$ satisfying $(\overline{\nu}_F)_\omega=\delta_{\xi(w)}$ and the equivariance property $b_{-1}\xi(T^{-1}w)=\xi(w)$. Recall also (see \S \ref{subsub.iwasawa}) that there exists an $H$-equivariant map $H/P \to \P(V)$ given by $hP \mapsto hV^N$ where $N$ is the unipotent radical of $P$. The image of $\overline{\nu}_F$ under this map is the unique $\mu$-stationary and proximal measure on $\P(V)$. It follows that the line $\R v_w$ is the image of $\xi(w)$ under the map $hP \mapsto hV^N$. Therefore, Lemma \ref{lemma.iwasawa.norm} implies that we have $\overline{\chi}(\overline{\alpha}(h,\xi(w)))=\log \frac{\|h v_w\|}{\|v_w\|}$ for $\mu^\Z$-a.e.~ $w \in H^\Z$. 

Given $\delta>0$, let $K(\delta)$ and $C(\delta)>1$ be as given by Lemma \ref{lemma.comparison.dyn.norm}, $C(\delta)$ increased if necessary to satisfy $2 \log C(\delta) \leq C(\delta)$. Then, if $w$ and $T^n w $ belong to $K(\delta)$, since $a^n v_w=v_{T^nw}$, by Lemma \ref{lemma.comparison.dyn.norm}, both $\frac{\|a^n v_w\|_{T^nw}}{\|a^nv_w\|}$ and $\frac{\|v_w\|_w}{\|v_w\|}$ belong to $[1/C(\delta), C(\delta)]$. The corollary follows.
\end{proof}

\subsubsection{Divergence estimates}

We also need the following lemma which essentially follows from Oseledets' theorem and Lemma \ref{lemma.comparison.dyn.norm}.

\begin{lemma}\label{lemma.div.est} \cite[Lemma 3.5]{eskin-lindenstraus.short}
For every $\delta>0$ and $t_0 \in \N$, there exists a compact subset $K'(\delta,t_0)=K$ of $H^\Z$ with $\mu^\Z(K)>1-\delta/10$ and a constant $C=C(\delta,t_0)>0$ with the following property: for every $w \in K$, $w' \in W^-_1(w) \cap K$ and $t>0$ such that $T^t w \in K$ and $T^tw' \in T^{[-t_0,t_0]}K$, we have
$$
|\lambda_1(w,t)-\lambda_1(w',t)| \leq C.$$
\end{lemma}

Here $W^-_1(w)$ is the local stable leaf of $w$ in the shift space $H^\Z$, i.e.~ $W^-_1(w)=\{w' \in H^\Z : w'_k=w_k,  \;  \forall k \geq 0\}$.

\subsubsection{Non-degeneracy of the stationary measure on projective space}

\begin{theorem}\cite[Theorem 3.1]{bougerol.lacroix}\label{thm.random.matrix.product}
Let $\mu$ be a Zariski-dense probability measure on a linear semisimple $\R$-split group $H$ and let $V$ be an irreducible algebraic representation of $H$. Then, for $\mu^{-\N^\ast}$-a.e.~ $b=(b_{-1},\ldots)$, any limit point $\hat{\pi}_b$ of the sequence $\frac{b_{-1}\ldots b_{-n}}{\| b_{-1}\ldots b_{-n}\|}$ in $\Endo(V)$ has rank one and the same image. Moreover, for any hyperplane $W<V$, the set of $b \in H^{-\N^\ast}$ such that $Im(\hat{\pi}_b) \in W$ has zero measure.
\end{theorem}

The image of any such limit point will be denoted $\R v_b$, i.e.~ $v_b \in V$ denotes a choice of a non-zero unit vector (for the norm $\|.\|$) in the image line.

We record the following statement which follows from Theorem \ref{thm.random.matrix.product}.

\begin{lemma}\label{lemma.random.lin.form}
For $\mu^\N$-a.e.~ $a \in H^\N$, there exists a linear form $\varphi_a$ of unit norm on $V$. For every such $a$ and for every $\delta>0$, there exist $\epsilon>0$ and a compact subset $K_a(\delta)$ of $H^{-\N^\ast}$ with $\mu^{-\N^\ast}(K_a(\delta))>1-\delta/10$ with the property that if $b, b' \in K_a(\delta)$, we have $|\varphi_a(v_b)|>\epsilon$ and
$$
\lim_{n \to \infty} \frac{\|a_{n}\ldots a_0 v_b\|}{\|a_{n}\ldots a_0 v_{b'}\|} =\frac{|\varphi_a(v_b)|}{|\varphi_a(v_{b'})|}.
$$
Moreover, for any linear form $\varphi$ on $V$, the set of $a\in H^\N$ such that $\varphi_a\in \R \varphi$ has zero measure.
\end{lemma}


\begin{proof}
Applying Theorem \ref{thm.random.matrix.product} to the sequence of transposes of $a_i$'s, we get that for $\mu^\N$-a.e.~ $a=(a_0,a_1,\ldots)$, any limit point of the sequence $\frac{a_n \ldots a_0}{\|a_n \ldots a_0\|}=(\frac{a_0^t \ldots a_n^t}{\|a_0^t \ldots a_n^t\|})^t$ of linear transformations is a rank-one linear map; denote it by $\pi_a$. Note that the kernel of $\pi_a$ does not depend on the choice of the limit rank-one transformation. We then define $\varphi_a$ linear form given by orthogonal projection onto the line orthogonal to $\ker\pi_a$. By the transpose relation, indeed we have
\[\ker \varphi_a=\ker\pi_a=(\im{\hat{\pi}_a})^\perp, \]
where $\hat{\pi}_a$ is a limit point of the sequence $\frac{a_0^t \ldots a_n^t}{\|a_0^t \ldots a_n^t\|}$ and $\pi_a=\hat{\pi}_a^t$. Due to the last claim of Theorem \ref{thm.random.matrix.product}, we obtain the last claim of this lemma.

Applying Theorem \ref{thm.random.matrix.product} to $b$, then the last claim of Theorem \ref{thm.random.matrix.product} implies that the $\mu^{-\N^\ast}$-measure of $b$'s such that $\ker \pi_a$ contains $v_b$ is zero. Therefore, given a typical $a$ (i.e.~ in a set of full measure) and $\delta$, there exists a compact set $K_a(\delta)$ in $H^{-\N^\ast}$ such that $\mu^{-\N^\ast}(K_a(\delta))>1-\delta/10$ and for every $b \in K_a(\delta)$, we have $d(\ker \pi_a, v_b)>\epsilon$, where $d$ denotes the projective distance induced by $\|.\|$. That is, $d(\ker \pi_a, v_b)=\frac{|\varphi_a(v_b)|}{\|\varphi_a\|\|v_b\|}$.
\end{proof}


\subsubsection{Relative density of typical points}
We also need one more lemma that will be used to spread the initial invariance obtained via the drift argument for a word $\hat{\omega}$ to a set of words with positive measure.

\begin{lemma}\label{lemma.relative.density.typical.orbit}
Let $X$ be a separable metric space, $m$ a Borel probability measure on $X$ and $T$ a measurable measure-preserving and ergodic transformation $X \to X$. Then, for any measurable subset $K$ of $X$ with positive $m$-measure, there exists a conull subset $\dot{K}$ 
such that
\[K_1:=\{x\in K : \overline{K\cap T^{\N}x}\supset \dot{K} \}.\]
 is conull subset of $K$.
\end{lemma}
\begin{proof}

Consider the induced system $(K_0,m|_{K_0}, T^K)$ where $T^K$ is the first return map to the set $K$ and $K_0$ is the conull subset of $K$ on which the points are infinitely recurrent (Poincar\'e recurrence). 
By ergodicity of $T$, we know that $T^K$ is also ergodic with respect to the measure $m|_{K_0}$. By Birkhoff's theorem, we know that for $m|_{K_0}$ a.e.~ $x \in K_0$ the orbit $\{(T^K)^nx\}_{n\in\N}$ equidistributes to the measure $m|_{K_0}$. 
So for $m|_{K_0}$ a.e.~ $x$, we have 
\[ \overline{K\cap T^{\N}x}= \overline{ (T^K)^{\N}x}\supset \Supp{m|_{K_0}}=:\dot{K}. \]
The proof is complete.
\end{proof}

We are now ready to start

\begin{proof}[Proof of Theorem \ref{thm.measure.class.geod}]

\textbf{Choosing parameters and sets}:
Let $\delta \in (0,1/10)$ be a small enough positive constant. Let $\Psi$ be a $T$-invariant full measure set contained in the intersection of the full-measure subset of $H^\Z$ given by Proposition \ref{prop.dynnorm} with the full-measure subset on which the map $\omega \mapsto \nu_\omega$ is defined. Denote by $C(\delta)$ the constant given by Corollary \ref{corol.alpha.close.to.dynamical.cocycle} and $K(\delta) \subseteq \Psi$ a compact set chosen using the same corollary satisfying $\mu^\Z(K(\delta))>1-\delta/20$. Fix a compact subset $K_{cont}$ of $H^\Z$ of $\mu^\Z$-measure $>1-\delta/20$ on which the map from $w\in H^\Z$ to $\nu_w$ in the space of probability measures on $F$ is continuous.
Let $$K_0(\delta)=K(\delta) \cap K_{cont} \cap K'(\delta,1),$$
where $K'(\delta,1)$ is the compact subset of $\Psi$ obtained from Lemma \ref{lemma.div.est} and let $C$ be the positive constant given by the same lemma.

\bigskip

Now applying Lemma \ref{lemma.relative.density.typical.orbit} to the shift system $X=H^\Z$ and $m=\mu^\Z$ with $K=K_{0}(\delta)$, by regularity of $\mu^\Z$, we can find a compact subset $K_0'(\delta)$ of $K_{0}(\delta)$ with $$
\mu^\Z(K_{0}(\delta))-\mu^\Z(K_0'(\delta))<\delta/20
$$
such that for every $w \in K_0'(\delta)$, the closure of the intersection of the $T$-orbit of $w$ with $K_0(\delta)$ contains a $\mu^\Z$-conull subset of $K_0(\delta)$. Finally, fix a compact $\underline{K}$ of $H$ with sufficiently large $\mu$-measure, so that 
$$
\hat{\underline{K}}:=\{w \in H^\Z : (w_{-C_2},\ldots,w_0,\ldots,w_{C_2}) \in \underline{K}^{2C_2+1}\}
$$
has $\mu^\Z$-measure $>1-\delta/20$, where $C_2=[\kappa(2\kappa +C)]+1$, with $\kappa$ as given by Proposition \ref{prop.dynnorm}. We now let $$K_{0}''(\delta)=K_{0}(\delta) \cap K_0'(\delta) \cap \hat{\underline{K}}.$$

\bigskip

Let $N(\delta) \in \N$ be a constant so that there exists a compact subset  
\begin{equation}\label{equ:Kgen}
K^{gen}(\delta) \subset \{w \in H^\Z : \frac{1}{n}\#\{k=1,\ldots,n : T^k w \in K_0''(\delta)\}>1-\delta/2, \; \forall n \geq N(\delta)\}
\end{equation}
with $\mu^\Z$-measure $\geq 1-\delta/10$. The existence of this set is ensured thanks to Birkhoff's ergodic theorem. 
Indeed, due to Birkhoff's ergodic theorem, we obtain for $\mu^{\Z}$ a.e.~ $w\in H^{\Z}$,
\[\lim_{n\rightarrow\infty}\frac{1}{n}\#\{k=1,\cdots,n : T^kw\in K_0''(\delta) \}\rightarrow \mu^{\Z}(K_0''(\delta))>1-4\delta/10. \]
We can therefore find a large constant $N(\delta)$ such that \eqref{equ:Kgen} holds.
Let $$K_{00}(\delta)=K_0''(\delta) \cap K^{gen}(\delta).$$

\bigskip

For an element $a \in H^\N$ and a subset $K$ of $H^\Z$, let $K_a^-$ denote the set $\{b \in H^{-\N^\ast} : (b,a) \in K\}$. By Markov's inequality, the set $$K^+=\{a \in H^{\N} \, | \, \mu^{-\N^{\ast}}(K_a^-) \geq 1-\sqrt{\delta}\}$$  satisfies $\mu^{\N}(K^{+}) \geq 1- \sqrt{\delta}$, if $\mu^\Z (K)>1-\delta$. Specializing to $K=K_{00}(\delta)$, we fix two elements $a,a' \in K^+$, so that the set 
$$K_{a,a'}^-:=K_a^- \cap K_{a'}^-$$ 
has $\mu^{-\N^\ast}$-measure larger than $1-2\sqrt{\delta}$.

\bigskip

For each $t \in \N$ and $w \in \Psi$, let $n_t(w)=\min\{n : \lambda_1(w,n) \geq t\}$. We then have 
\begin{equation}\label{equ:stopping}
|\lambda_1(w, n_t(w))-t| \leq \kappa,
\end{equation}
which is due to Proposition \ref{prop.dynnorm}.

\bigskip

\textbf{Drift argument}:
The output of this part is
\begin{proposition}\label{prop:drift}
For two futures $a,a'\in K_{00}^+(\delta)$ and two pasts $b,b' \in K_{a,a'}^-$, 
there exist sequences of natural numbers $m_\ell, m_\ell'\to\infty$ as $\ell\to\infty$, a point $\hat{\omega}\in K_0''(\delta)$ and $s(b,a,a'),\ s(b',a,a')\in D^{\pm}$ such that
$\nu_{T^{m_\ell}(b,a)}\to \nu_{\hat\omega}$ and
$$\nu_{\hat{\omega}}=s(b,a,a')^{-1}s(b',a,a')\nu_{\hat{\omega}},$$ 
where the element $s(b,a,a')$ is given by
\[ 
s(b,a,a')=\lim_{\ell\to\infty}\alpha(a_{m_{{\ell}}'}' \ldots a_0', \xi(b))\alpha(a_{m_{\ell}} \ldots a_0, \xi(b))^{-1},
 \]
and similarly for $s(b',a,a')$.
\end{proposition}

We start the drift argument here. Set $w=(b,a)$, $w'=(b,a')$, $w''=(b',a)$, and $w'''=(b',a')$. 
\begin{claim}
There are constants $p(\delta)$ with $p(\delta)=O(\delta)$ as $\delta \to 0$ and $N_1(\delta) \in \N$ such that for any $\zeta, \zeta' \in K^{gen}(\delta) \cap \Psi$ and $T \geq N_1(\delta)$, we have 
\begin{equation}\label{eq.often.common.rec}
    \#\{t=1,\ldots, T : T^{n_t(\zeta)}\zeta' \notin K_0''(\delta)\}<p(\delta)T.
\end{equation}
\end{claim}

\begin{proof}
By \eqref{equ:Kgen}, we have for $n\geq N(\delta)$
\[\#\{k=1,\cdots, n : T^k\zeta'\notin K_0''(\delta) \}<\delta n/2. \]
By Lipschitz property \eqref{eq.lip.dynnorm} for $\zeta\in\Psi$, we have $n_T(\zeta)\in [1/\kappa,\kappa] T$. Hence for $T>N_1(\delta)=\kappa N(\delta)$, using Lipschitz property \eqref{eq.lip.dynnorm}, we have
\[\#\{t=1,\ldots, T : T^{n_t(\zeta)}\zeta' \notin K_0''(\delta)\}\leq \kappa \#\{k=1,\ldots, n_T(\zeta) : T^{k}\zeta' \notin K_0''(\delta)\}<\delta\kappa^2T/2, \]
proving the claim \eqref{eq.often.common.rec} with $p(\delta)=\delta \kappa^2/2$.
\end{proof}

Therefore, choosing $\delta>0$ small enough so that to have $p(\delta)<1/16$ and applying \eqref{eq.often.common.rec} with all possible choices of $\omega,\omega' \in \{w, w', w'',w'''\}$, we find a sequence of positive integers $t_\ell$ tending to infinity as $\ell \to \infty$ and such that for every $\ell \in \N$, and $\omega,\omega' \in \{w, w', w'',w'''\}$, we have 
\begin{equation}\label{eq.common.return.times}
T^{n_{t_\ell}(\omega)}\omega'  \in K_0''(\delta).
\end{equation}

\begin{claim}
For every $\ell \in \N$, 
\begin{equation}\label{eq.control.nt.difference}
|n_{t_\ell}(w)-n_{t_\ell}(w'')| \, \, \, \text{and} \, \, \, |n_{t_\ell}(w')-n_{t_\ell}(w''')| \, \, \, \text{are bounded above by} \, \, \kappa(2\kappa+C),
\end{equation}
where $C=C(\delta,1)$ is the constant given by Lemma \ref{lemma.div.est}.
\end{claim}

\begin{proof}
Indeed, by construction, we have $w'' \in W^-_1(w)$, $w''' \in W^-_1(w')$. Moreover, thanks to \eqref{eq.common.return.times} and the fact that $K_0''(\delta)$ is contained in $K'(\delta,1)$, we can apply Lemma \ref{lemma.div.est} and deduce that $$|\lambda_1(w,n_{t_\ell}(w))-\lambda_1(w'',n_{t_\ell}(w))| \leq C.$$ 
On the other hand, by \eqref{equ:stopping}, we have $$|\lambda_1(w,n_{t_\ell}(w))-t_\ell| \leq \kappa,\ |\lambda_1(w'',n_{t_\ell}(w''))-t_\ell| \leq \kappa,$$ so that $$|\lambda_1(w,n_{t_\ell}(w))-\lambda_1(w'',n_{t_\ell}(w''))| \leq 2\kappa.$$ This implies that $$|\lambda_1(w'',n_{t_\ell}(w''))-\lambda_1(w'',n_{t_\ell}(w))| \leq 2 \kappa+C.$$ 
Referring once more to the Lipschitz property \eqref{eq.lip.dynnorm}, we deduce that $|n_{t_\ell}(w'')-n_{t_\ell}(w)| \leq \kappa(2\kappa+C)$ as claimed. Clearly, the same argument applies to $n_{t_\ell}(w')$ and $n_{t_\ell}(w''')$ proving \eqref{eq.control.nt.difference}.
\end{proof}

\bigskip

It then follows by \eqref{eq.comparison.dyn.norm}, construction of $K_0''(\delta)$, and \eqref{equ:stopping} that for every $\ell \in \N$, we have 
\begin{equation}\label{eq.bdd.drift.ww'}
\begin{split}
&|\overline{\chi}(\overline{\alpha}(a_{n_{t_\ell}(w')}' \ldots a_0', \xi(b))- \overline{\chi}(\overline{\alpha}(a_{n_{t_\ell}(w)} \ldots a_0, \xi(b))| \\
&\leq |\lambda_1(w',n_{t_\ell}(w'))-\lambda_1(w,n_{t_\ell}(w))|+2C(\delta) \leq 2 \kappa + 2 C(\delta),
\end{split}
\end{equation}
and similarly,
\begin{equation}\label{eq.bdd.drift.w''w'''}
|\overline{\chi}(\overline{\alpha}(a_{n_{t_\ell}(w''')}' \ldots a_0', \xi(b'))- \overline{\chi}(\overline{\alpha}(a_{n_{t_\ell}(w'')} \ldots a_0, \xi(b'))| \leq 2 \kappa + 2 C(\delta).
\end{equation}

\bigskip

Now thanks to the fact that $\nu_w=\nu_{w'}$ (since $w$ and $w'$ have the same past),
using the equivariance relation \eqref{eq.equiv.cocycle} at times $n_{t_\ell}(w)$ and $n_{t_{\ell}}(w')$, we get
\begin{equation}\label{eq.use.equiv.1}
\nu_{T^{n_{t_{\ell}}(w')}w'}=\underbrace{\left(\alpha(a_{n_{t_{\ell}}(w')}' \ldots a_0', \xi(b))\alpha(a_{n_{t_\ell}(w)} \ldots a_0, \xi(b))^{-1}\right)}_{=:D_\ell}\underbrace{\alpha(a_{n_{t_\ell}(w)} \ldots a_0, \xi(b)) \nu_w}_{\nu_{T^{n_{t_{\ell}}(w)}w}},
\end{equation}
and similarly,
\begin{equation}\label{eq.use.equiv.1'}
\nu_{T^{n_{t_{\ell}}(w''')}w'''}=\underbrace{\left(\alpha(a_{n_{t_{\ell}}(w''')}' \ldots a_0', \xi(b'))\alpha(a_{n_{t_\ell}(w'')}\ldots a_0, \xi(b'))^{-1}\right)}_{=:D'_\ell}\underbrace{\alpha(a_{n_{t_\ell}(w'')} \ldots a_0, \xi(b')) \nu_{w''}}_{\nu_{T^{n_{t_{\ell}}(w'')}w''}}.
\end{equation}

Here, thanks to, respectively \eqref{eq.bdd.drift.ww'} and \eqref{eq.bdd.drift.w''w'''}, the sequences $D_\ell$ and $D_\ell'$ are bounded. Moreover, by construction (see \eqref{eq.common.return.times}) $T^{n_{t_{\ell}}(\zeta')}\zeta$ belong to the compact continuity set $K_{cont}$ for every $\zeta,\zeta' \in \{w,w',w'',w'''\}$. In particular, there exists a subsequence of $t_\ell$ such that for any such $\zeta,\zeta'$ and for some $\hat{\omega} \in H^\Z$, we have
\begin{equation}\label{eq.use.continuity1}
T^{n_{t_\ell}(\zeta')}\zeta \to_{\ell \to \infty} \hat{\omega} \implies \nu_{T^{n_{t_{\ell}}(\zeta')}\zeta} \to_{\ell \to \infty} \nu_{\hat{\omega}}.
\end{equation} 

\bigskip


Now we redefine the time. Let $m_{t_\ell}(w'')=m_{t_\ell}(w)=\max\{n_{t_\ell}(w),n_{t_\ell}(w'') \}$ and similarily $m_{t_\ell}(w''')=m_{t_\ell}(w')=\max\{n_{t_\ell}(w'),n_{t_\ell}(w''') \}$.
Indeed, by construction of $n_{t_\ell}$'s and $m_{t_\ell}$'s, the convergence property is unaffected: see \eqref{eq.use.continuity1} and the choice of $\zeta,\zeta'\in\{w,w',w'',w''' \}$.
Moreover, thanks to the equivariance property, we still have the relations \eqref{eq.use.equiv.1} and \eqref{eq.use.equiv.1'} for these modified times $m_{t_\ell}$'s. Finally, the differences between $n_{t_\ell}(w)$ and $n_{t_\ell}(w'')$, and similarly, between $n_{t_\ell}(w')$ and $n_{t_\ell}(w''')$ are bounded, (see \eqref{eq.control.nt.difference}). Since we have chosen $K_0''(\delta)$ so that it is contained in the set $ \hat{\underline{K}}$, the modified differences --- as appearing in \eqref{eq.use.equiv.1} and \eqref{eq.use.equiv.1'} after modifying the times --- $D_\ell$ and $D_\ell'$ are still bounded. 

\bigskip

Notice that since $w$ and $w''$; and similarly, $w'$ and $w'''$ have the same futures, for any sequence of $\ell$'s such that $T^{m_{t_\ell}(w)}w \to \hat{w}$ for some $\hat{w}$, we also have $T^{m_{t_\ell}(w'')}w''=T^{m_{t_\ell}(w)}w''\to \hat{w}$ (and similarly for the pair $w'$ and $w'''$). As a conclusion passing to a subsequence of $t_\ell$'s (that we still denote by $t_\ell$) so that we have\\[2pt] 
\indent \textbullet ${}$  $D_\ell \to s(b,a,a')$ for some $s(b,a,a') \in D^{\pm}$ and similarly, $D'_\ell \to s(b',a,a')$ for some $s(b',a,a') \in D^{\pm}$, and\\[2pt]
\indent \textbullet ${}$ $T^{m_{t_\ell}(w)}w \to \hat{w} \in K_0''(\delta)$ and $T^{m_{t_\ell}(w')}w' \to \hat{w}' \in K_0''(\delta)$,\\[2pt]
we deduce from \eqref{eq.use.equiv.1} and \eqref{eq.use.equiv.1'} that
$$
\nu_{\hat{w}'}=s(b,a,a')\nu_{\hat{w}} \quad \text{and} \quad  \nu_{\hat{w}'}=s(b',a,a')\nu_{\hat{w}},
$$
and hence we get
\begin{equation}\label{eq.take.square.if.pgl}
\nu_{\hat{w}}=s(b,a,a')^{-1}s(b',a,a') \nu_{\hat{w}}.
\end{equation}

By letting $m_\ell=m_{t_\ell}(w)$ and $m_\ell'=m_{t_\ell}(w')$, we obtain Proposition \ref{prop:drift} stated in the beginning of this part.

\bigskip

\textbf{From invariance of one typical point to the full set}: By the equivariance property and commutativity, for every $t \in \N$, the measure $\nu_{T^t \hat{w}}$ is also invariant by $s(b,a,a')^{-1}s(b',a,a')$. On the other hand, we have $\hat{w} \in K_0''(\delta)$ and recall that the latter set is contained in $K_0'(\delta)$. So letting $K_{acc}$ be the set of elements $\omega$ in $K_0(\delta)$ such that there exists a sequence $n_m \to \infty$ such that $T^{n_m}\hat{w} \in K_0(\delta)$ and $T^{n_m}\hat{w} \to \omega$, by the definition of $K_0'(\delta)$, we get that the $\mu^\Z$-measure of $K_{acc}$ is positive. Since $K_0(\delta)$ is contained in the continuity set, this implies that for every $\omega \in K_{acc}$, the measure $\nu_{\omega}$ is invariant by $s(b,a,a')^{-1}s(b',a,a')$. By ergodicity and commutativity (since the set of $\omega$ such that $\nu_{\omega}$ is invariant by an element of $D$ is shift-invariant), this entails that
$$
\nu_{\omega}=s(b,a,a')^{-1}s(b',a,a')\nu_{\omega} \quad \text{for $\mu^\Z$-a.e. $\omega \in H^\Z$}.    
$$

\bigskip

\textbf{Constructing arbitrary small drift}:
Since for $\mu^\Z$-a.e.~ $\omega$ the stability group of $\nu_{\omega}$ is closed, to prove the hypothesis of Claim 0 (and hence Theorem \ref{thm.measure.class.geod}), it suffices to find  sequences $\delta_n>0$, couples of futures $a_n,a'_n \in K^+:= K^+_{00}(\delta_n)$ and  couples of pasts $b_n,b'_n \in K_{a_n,a'_n}^{-}$ such that 
\begin{equation}\label{eq.completes.drift}
\id \neq (s(b_n,a_n,a'_n)^{-1}s(b'_n,a_n,a'_n))^2 \to \id
\end{equation}
as $n \to \infty$. Here we take square to make sure the invariance is in $D$ instead of $D^{\pm}$.

Recall that $a,a'$ are two different points in $K_{00}^+(\delta)$. Due to Lemma \ref{lemma.random.lin.form} and the set $K_{00}^+(\delta)$ having a positive measure, we can suppose that the corresponding linear forms $\varphi_a$ and $\varphi_{a'}$ are not colinear. The set $K_{a,a'}^-$ has measure greater than $1-2\sqrt{\delta}$. 
Now given $\delta'>0$, consider the compact set $K_{a,a'}(\delta'):=K_a(\delta')\cap K_{a'}(\delta')$ given by Lemma \ref{lemma.random.lin.form}. Clearly, if $\delta$ and $\delta'$ are small enough, the set $K_{a,a'}(\delta') \cap K_{a,a'}^-$ has positive measure, bounded below by $1-2\delta'-2\sqrt{\delta}$. 
On the other hand, by \textit{drift argument} (Proposition \ref{prop:drift}), for every $b,b' \in K_{a,a'}(\delta') \cap K_{a,a'}^-$, there exist sequences of natural numbers $m_\ell, m_\ell'$ tending to infinity as $\ell \to \infty$ such that
\begin{equation}\label{eq.cross.ratio.limit}
\overline{\chi}\left(\log(s(b,a,a')^{-1}s(b',a,a'))\right)=\lim_{\ell \to \infty} \log \left(\frac{\|a'_{m_{\ell}'} \ldots a'_0 v_{b'}\|\|a_{m_{\ell}} \ldots a_0 v_{b}\|}{\|a'_{m_{\ell}'} \ldots a'_0 v_{b}\|\|a_{m_{\ell}} \ldots a_0 v_{b'}\|} \right).
\end{equation}
By Lemma \ref{lemma.random.lin.form}, we have for two linear forms $\varphi=\varphi_a$ and $\varphi'=\varphi_{a'}$ of unit norm on $V$ that 
\begin{equation}\label{eq.loglin.form.ratio}
\overline{\chi}\left(\log(s(b,a,a')^{-1}s(b',a,a'))\right)= \log \frac{|\varphi'(v_{b'})\varphi(v_b)|}{|\varphi'(v_b)\varphi(v_{b'})|}
\end{equation}
and $|\varphi(v_b)|,|\varphi'(v_b)|>\epsilon'>0$, where $\epsilon'=\epsilon'(\delta')$ is given by Lemma \ref{lemma.random.lin.form}.

Let $\epsilon>0$ be given. 
Since $\mu^{-\N^\ast}(K_{a,a'}(\delta') \cap K_{a,a'}^-)>0$ and the Furstenberg measure is atomless, we can find two different points $b,b'$ in $K_{a,a'}(\delta') \cap K_{a,a'}^-$ with $v_b\wedge v_{b'}\neq 0$ and
$d(v_b,v_{b'})<\epsilon$. 

\begin{claim}
      If $2\epsilon<(\epsilon')^2$, the drift element associated to $a,a',b,b'$ (as in  \eqref{eq.loglin.form.ratio}) is non-trivial and has size $O_{\epsilon'}(\epsilon)$. 
\end{claim}
\begin{proof}
This is because
\[\frac{\varphi(v_b)\varphi'(v_{b'})}{\varphi(v_{b'})\varphi'(v_b)}-1=\frac{(\varphi,\varphi')(v_{b'}\wedge v_b)}{\varphi(v_{b'})\varphi'(v_b)}, \]
where $(\varphi,\varphi')$ is a linear form on $\wedge^2 V$ given by 
\[(\varphi,\varphi')(v\wedge v')=\varphi(v)\varphi'(v')-\varphi(v')\varphi'(v). \]
Non-triviality comes from the choice of $a,a'$ and $b,b'$, that $\varphi$ and $\varphi'$ are not colinear and $v_b\wedge v_{b'}\neq 0$.

By taking $\epsilon<(\epsilon')^2/2$, we have
\[ \frac{|(\varphi,\varphi')(v_{b'}\wedge v_b)|}{|\varphi(v_{b'})\varphi'(v_b)|}\leq \frac{\|v_{b'}\wedge v_b\|}{| \varphi(v_{b'})\varphi'(v_b)|}\leq \epsilon/(\epsilon')^2<1/2. \]
Applying the inequality $|\log(1+t)|\leq 2|t|$ for $|t|<1/2$, we obtain
\[\left|\log \frac{\varphi'(v_{b'})\varphi(v_b)}{\varphi'(v_b)\varphi(v_{b'})}\right|\leq 2 \frac{|(\varphi,\varphi')(v_{b'}\wedge v_b)|}{|\varphi(v_{b'})\varphi'(v_b)|}\leq 2\epsilon/(\epsilon')^2. \]
The proof of the claim is complete.
\end{proof}
\bigskip

Fixing $\epsilon'>0$ and choosing $\epsilon$ arbitrarily small --- i.e.~ taking a sequence $\epsilon_n \to 0$ and associated couples $b_n,b'_n \in K_{a,a'}(\delta') \cap K_{a,a'}^-$ --- we obtain \eqref{eq.completes.drift} and conclude the proof.
\end{proof}

\subsection{Case 2.3: The remaining case}\label{subsec.remaining.case}

In this part, we will restrict ourselves to a slightly more specific situation; we will assume that the ambient group $G$ is $\PGL_n(\R)$; the subgroups $H, Q,R,R_0$ have the same meaning as before. 
The group $S = Q/R_0$ is a quotient of a product of $\PGL_{k_i}(\R)$'s. 

We are in Case 2.3, so we suppose $H$ is positioned so that $Q_H:=Q \cap H$ is a parabolic subgroup of $H$ and that $Q_H^\circ \cap R_0$ is trivial. 
In light of Proposition \ref{prop.decomposable} and Definition \ref{def.decomposable} of a decomposable action, it might be tempting at first sight to think that in Case 2.3, the morphism extends and we are in the decomposable situation. However, it turns out this is not the case and whether the morphism can extend depends for example on the irreducibility of the action of $H$ on $\P(\R^n)$. We signal at this point that in this paper we are not able to get a characterization of when we are in the decomposable case; as we shall see (Case 2.3.a), if $H$ acts projectively irreducibly on $\R^n$, we will be able to ensure this. Without this irreducibility assumption (Case 2.3.b), the description of what may happen is widely open; we content with some examples.


\subsubsection{Case 2.3.a: Irreducible $H$}

In this case, the irreducibility of $H$ implies that there is a unique $H$-compact orbit $\calC$ on the flag variety $G/Q$.

\begin{proof}[Proof of Theorem \ref{thm.irreducible.H.decompsable}]
If $R_0=Q$, then the fiber is trivial and we are in Case 2.1.

Otherwise, $S\simeq Q/R_0$ is a nontrivial semisimple group. In this situation, one can verify directly that $Q_H\cap R_0$ is trivial. For example by the explicit computation given in the following proof. So we are in Case 2.3.

Let $D<Q_H$ be a rank-one $\R$-split torus in $Q_H<H$ and $U$ be the unipotent radical of $Q_H$. We denote by $\mathfrak{u}$, $\mathfrak{d}$, and $\mathfrak{h}$ the Lie algebras of $U$, $D$, and $H$, respectively. Fix a Weyl chamber $\mathfrak{d}^+$ in $\mathfrak{d}$ and two elements $x \in \mathfrak{d}^+$ and $e \in \mathfrak{u}$ such that $[x,e]=2e$, where $[.,.]$ denotes the Lie bracket in $\mathfrak{h}$. We consider the Lie algebra representation of $\mathfrak{h}$ induced by the irreducible representation of $H$ coming from the embedding $H < \PGL_n(\R)$. By the representation theory of $\mathfrak{sl}_2(\R)$, the space $\R^n$ writes as a sum of a string of one-dimensional weight spaces of $\mathfrak{d}$, we denote them by $V_1,V_2,\ldots,V_n$. They are ordered in increasing order with respect to the order on weights of $D$ coming from the choice of $\mathfrak{d}^+$. The elements of $\mathfrak{u}$ act as raising operators, i.e.~ for any non-zero $e' \in \mathfrak{u}$, we have $e' V_i =V_{i+1}$ if $i\neq n$ and $e'V_n=0$.

Let $W_1<W_2<\ldots<W_k=\R^n$ be the maximal flag preserved by $Q$. Since the diagonal subgroup $D$ is contained in $Q_H$, each space $W_i$ is also preserved by $\mathfrak{d}$ and hence each $W_i$ is a sum of the weight spaces $V_i$'s. Moreover, since $U$ is also contained in $Q$ (and hence preserves $W_i$'s) and $\mathfrak{u}$ acts as raising operator for $\mathfrak{d^+}$ in the Lie algebra representation, it follows that $W_i=V_n \oplus \ldots \oplus V_{n-k_i+1}$, where $k_1<k_2<\ldots<k_j=n$ are the dimensions of $W_1,W_2,\ldots$ respectively. We also set $k_0=0$ and set $m_i=k_i-k_{i-1}$ for $i=1,\ldots,k$. The group $S$ is then a quotient of the product $\prod_{i=1}^j S_i$ where $S_i \simeq \PGL_{m_i}(\R)$. The product $\Pi_i S_i \simeq \Pi_i\PGL_{m_i}(\R)=Q/(R_0')$, where $R_0'$ is the solvable radical of $Q$. We have a natural projection from $\prod_i S_i=Q/(R_0')\to S=Q/R_0$. The projection $Q_H$ to $S=Q/R_0$ factors through $\prod_i S_i=Q/(R_0')$. Therefore, to extend the morphism to $S$, we only need to extend the morphism from $Q_H$ to $\prod_i S_i$.

Let $\mathfrak{s}_i$ be the Lie algebra $S_i$'s. The Lie algebra morphism from the Lie algebra of $Q_H$ to $\mathfrak{s}_i$ coming from the morphism $Q_H \to \prod_i S_i$ is the morphism obtained by extending
$$
x \mapsto \begin{pmatrix}
m_i-1 & 0 & \cdots  & & & 0 \\
0 & m_i-3 &  &  & & \vdots   \\
\vdots &   &   \ddots & &  \\
 & &  &  &   -m_i+3 & 0  \\
0 &  &  &  & 0 & -m_i+1
\end{pmatrix}_{m_i \times m_i}$$
$$  e \mapsto \begin{pmatrix}
0 & 1 & 0 & \cdots & 0 \\
 & 0 & 1 &  &  \\
 &  & \ddots  & \ddots &  \\
 & & &  0 & 1\\
0 & &  & \cdots  & 0 &  \\
\end{pmatrix}_{m_i \times m_i}\\
.
$$
To extend this morphism to $\mathfrak{h} \to \mathfrak{s}_i$, let $f$ be an element of $\mathfrak{h}$ so that $(e,x,f)$ is an $\mathfrak{sl}_2$-triple, i.e.~ $[x,f]=-2f$ and $[e,f]=x$. Mapping the element $f$ to the element
$$
\begin{pmatrix}
0 &  &  &  &  \\
(m_i-1) & 0 &  &  &  \\
 & 2(m_i-2)  & 0  &  &  \\
& & \ddots &  \ddots & & \\
 & & & (m_i-2)2&  0 &  & \\
 & & & & (m_i-1)   & 0  \\
\end{pmatrix}_{m_i \times m_i}
$$
of $\mathfrak{s}_i=\mathfrak{pgl}_{m_i}(\R)$, a direct calculation (see e.g.~ \cite[\S 3.7]{neil-ginzburg.book})
shows that we obtain a Lie algebra morphism $\mathfrak{h} \to \mathfrak{s}_i$ for each $i=1,\ldots,j$. We hence get a morphism $\mathfrak{h} \to \bigoplus_i\mathfrak{s}_i$ which gives rise to an algebraic morphism $H \to \prod_i S_i$ extending the initial morphism $Q_H \to \prod_i S_i$. (For the $\PGL_2(\R)$ case, notice that an irreducible algebraic representation from $\SL_2(\R)$ to $\PGL_{m_i}(\R)$ always induces a representation of $\PGL_2(\R)$)

Therefore Proposition \ref{prop.decomposable} yields that the $H$-action on $X_\mathcal{C}$ is decomposable. The last assertion then follows from Proposition \ref{prop.decomposable.measure.class} and uniqueness of the $\mu$-stationary probability measure (the Furstenberg measure) on $H/P$.
\end{proof}

We single out the following consequence which gives a generalization (and an explanation) of the phenomenon of embedding of the Furstenberg boundary in the fibre bundle $X$. This phenomenon is discovered in the work of Sargent--Shapira \cite{sargent-shapira} when $X$ is the space of $2$-lattices inside $\R^3$. 
\begin{example}[Rank-$k$ lattices in $n$-space]
Let $G=\PGL_n(\R)$ and $Q$ the stabilizer of a $k$-space $W$ in $\R^{n}$. 
Let $R$ be the stabilizer in $Q$ of the homothety class of a lattice in $W$, and $R_0$ the connected component of $R$. In this case, the bundle $G/R$ over $G/Q$ will be denoted as $X_{n,k}$. It is actually the space of homothety-equivalence classes of rank-$k$ lattices in $\R^n$. Recall that $H$ is a copy of $\SL_2(\R)$ or $\PGL_2(\R)$ acting irreducibly on $\P(\R^{n})$ and $\mathcal{C} \subset G/Q$ is the unique compact $H$-orbit in $G/Q$. That is $\calC=HQ\subset G/Q$. It is then easy to check that we are in the setting of Theorem \ref{thm.irreducible.H.decompsable} and therefore we get a trivialization $(X_{n,k})_\mathcal{C} \overset{\phi}{\simeq} H/Q_H \times S/\Lambda$, where $S=\PGL_k(\R)$, and $\Lambda=\PGL_k(\Z)$, such that the associated cocycle $H \times H/Q_H \to S$ is morphism-type, i.e.~ it does not depend on the $H/Q_H$ coordinate (in particular it is a morphism $\rho: H \to S$). In the statement below, the $\mu$-action on $S/\Lambda$ is defined via $\rho$. 
\end{example}


\begin{corollary}
Keep the above setting. In particular, let $H$ be an algebraic subgroup of $G$ isomorphic to $\SL_2(\R)$ or $\PGL_2(\R)$ and acting irreducibly on $\P(\R^n)$. Let $(X_{n,k})_\calC$ be the sub-bundle of $X_{n,k}$ over the base $\calC\subset G/Q$. Then, we have
\[P_\mu^{erg}((X_{n,k})_\calC)\simeq P_\mu^{erg}(S/\Lambda). \]

\end{corollary}

\smallskip

\subsubsection{Case 2.3.b: Reducible $H$}

Below, we give an example for Case 2.3.b and justify that for this example it is not possible to extend the morphism $Q_H \to S$.

\begin{example}\label{ex.to.be.treated}
Let $G=\PGL_4(\R)$ and $Q$ be the parabolic subgroup given by the stabilizer of the 3-plane generated by the standard basis vectors $\{e_1,e_2,e_3\}$. We take $R_0$ to be the solvable radical of $Q$ and $R$ to be the stabilizer of the $3$-lattice generated by $\{e_1,e_2,e_3\}$. Finally, we take $H$ to be the copy of $\PGL_2(\R)$ in $G$ given by
\begin{equation}\label{eq.2.3.b.embedding}
\left \{ \begin{pmatrix}
a^2 & ab & 0 & b^2 \\
2ac & ad+bc & 0 & 2bd \\
0 & 0 & 1 & 0 \\
c^2 & cd & 0 & d^2 \\
\end{pmatrix} | \begin{pmatrix}
a & b \\
c & d 
\end{pmatrix} \in \SL_2^{\pm}(\R)
\right \}.
\end{equation}
We claim that this configuration falls into Case 2.3.b. Indeed, the intersection $Q_H=Q \cap H$ is given by the image of the upper-triangular subgroup of $\PGL_2(\R)$ in the embedding \eqref{eq.2.3.b.embedding} described above; in other words
$$
Q_H=
\left \{ \begin{pmatrix}
a^2 & ab & 0 & b^2 \\
0 & \pm 1 & 0 & \pm 2ba^{-1} \\
0 & 0 & 1 & 0 \\
0 & 0 & 0 & a^{-2} \\
\end{pmatrix} | \; a \neq 0
\right \}.
$$
So $Q_H$ is a parabolic subgroup of $H$ and therefore we are in Case 2. It is easy to see that intersection of $Q_H \cap R_0$ is trivial, hence we are in Case 2.3. Finally, clearly the $H$-representation described in \eqref{eq.2.3.b.embedding} is not irreducible justifying the claim.

Now note that $S=Q/R_0$ is the group $\PGL_3(\R)$ and the projection $Q \to S$ is given by the projectivization of the top-left 3-by-3 block in $Q$. It follows that the morphism $Q_H \to S$ is given by 
\begin{equation}\label{eq.not.extend}
\PGL_2(\R) \ni \begin{pmatrix}
a & b \\
0 & \pm a^{-1} 
\end{pmatrix} \mapsto \begin{pmatrix}
a & b & 0 \\
0 & \pm a^{-1} & 0 \\ 
0 & 0 & a^{-1}
\end{pmatrix} \in \PGL_3(\R) 
\end{equation}
However, it is not hard to see that the morphism \eqref{eq.not.extend} from the upper-triangular subgroup of $\PGL_2(\R)$ to $\PGL_3(\R)$ is not the restriction of a morphism $\PGL_2(\R) \to \PGL_3(\R)$. One can either use the classification of $\SL_2(\R)$-representations to see this or otherwise verify this claim by direct computation: note by $x$ and $e$ a pair of Lie algebra elements of $\mathfrak{sl}_2(\R)$ in Lie algebra of the upper-triangular group satisfying $[x,e]=2e$. Let $\overline{x}$ and $\overline{e}$ be their images in $\mathfrak{pgl}_3(\R)$ under the Lie algebra representation induced by  \eqref{eq.not.extend}. Now one checks by direct computation that it is not possible to find an element $\overline{f}$ in $\mathfrak{pgl}_3(\R)$ satisfying $[\overline{x},\overline{f}]=-2\overline{f}$ and $[\overline{e},\overline{f}]=\overline{x}$.


\end{example}

\section{$\SL_2(\R)$-Zariski closure: equidistribution}\label{sec.equidist}

In this part, we study equidistribution of the averaged measure $\frac{1}{n}\sum_{1\leq k\leq n}\mu^{*k}*\delta_x$ for $x$ inside the bundle $X_\calC$.
In fact, as we start by briefly explaining in \S \ref{subsec.equidist.trivial} below, all of them except the diagonal fibre action (Case 2.2) boils down to the corresponding results of Benoist--Quint \cite{BQ1,BQ3}. The part \S \ref{subsec.equidist.diagonal} is devoted to the diagonal fibre action case.

\subsection{Equidistribution from Benoist--Quint}\label{subsec.equidist.trivial}

In each case below, we keep the corresponding assumptions from \S \ref{sub.base.and.cases}.

\subsubsection{Case 1 (Dirac Base)} Recall that Case 1 corresponds to the situation when the acting group $H$ is contained in the parabolic $Q$ of $G$. As explained in \S \ref{subsub.dirac.base}, it follows that $H$ fixes a point in $G/Q$ and hence stabilizes the fibre above the fixed point. Therefore, up to conjugating $Q$, we are left with studying the associated $\mu$-random walk on the fibre $S/\Lambda$, where the probability measure $\mu$ is seen as a Zariski-dense measure in a copy of $\SL_2(\R)$ in the subgroup of $S$. This is then a particular situation of the setting treated in Benoist--Quint's work \cite{BQ1,BQ3}. Consequently, the corresponding equidistribution results apply. We do not state the result here as it would be a repetition. We refer the reader to the more recent \cite[Theorem 1.5]{prohaska-sert-shi}, where the compact support assumption of \cite{BQ3} is relaxed to finite exponential moment.

\subsubsection{Case 2.1 (Trivial fiber action)}\label{subsub.equidist.trivial.fibre}

Recall from Proposition \ref{prop.trivial.fibre.measure.class} that in this case the $H$-action on $X_\mathcal{C}$ is decomposable with trivial morphism, i.e.~ there exists a standard trivialization $X \simeq G/Q \times Q/R$ for which the associated cocycle restricted to $\mathcal{C}$ is the trivial morphism. Therefore in this case we have $\mu^{\ast k} \ast \delta_{(\theta,f)}=\int \delta_{g \theta} d\mu^{\ast k}(g) \otimes \delta_f$, in other words, the equidistribution problem is only the one in $\mathcal{C} \subseteq G/Q$. It is well-known that by spectral gap property we have the convergence $\int \delta_{g \theta} d\mu^{\ast k}(g) \to \overline{\nu}_F$ moreover with exponential speed estimates with respect to a class of H\"{o}lder functions. We omit the statement to avoid repetition; see \cite[Ch.~ V, Theorem 4.3]{bougerol.lacroix}.

\subsubsection{Case 2.3.a (Irreducible $H$-action)}

\begin{proposition}\label{prop.equidist.H.irred}
Keep the setting of Theorem \ref{thm.irreducible.H.decompsable} and suppose moreover that the measure $\mu$ on $H$ has finite exponential moment. Then, there exists a standard trivialization $X \simeq G/Q \times S/\Lambda$ such that for every $x \in X_\mathcal{C}$, the limit as $n$ tends to infinity of $\frac{1}{n}\sum_{k=1}^n \mu^{\ast k}\ast \delta_x$ exists and equals to a product $\overline{\nu}_F \otimes \nu^F$, where $\overline{\nu}_F$ is the Furstenberg measure on $H/Q_H$ and $\nu^F$ is a homogeneous probability measure on $S/\Lambda$.
\end{proposition}

As we shall see, the statement follows as a consequence of the decomposability of $H$-action (Theorem \ref{thm.irreducible.H.decompsable}), Benoist--Quint \cite{BQ1,BQ3} equidistribution results. We note however that we do not treat the question of equidistribution of trajectories of points $x \in X \setminus X_\mathcal{C}$. For such points, already in the level of the base space $G/Q$, the corresponding equidistribution question does not seem to be well-understood in all cases (cf.~ \cite{BQ.compositio}).  

\begin{remark}
    The conclusion of Proposition \ref{prop.equidist.H.irred} also holds if we replace the Ces\`{a}ro average $\frac{1}{n}\sum_{k=1}^n \mu^{k}\ast \delta_x $ by the sequence of empirical measures. More precisely, for every $x \in X_\mathcal{C}$, for $\mu^{\N}$-a.e.~ $a \in H^{\N}$, the sequence $\frac{1}{n}\sum_{k=0}^{n-1} \delta_{a_k\ldots a_0 x}$ converges to a product measure of the same form as in  Theorem \ref{thm.irreducible.H.decompsable}. This follows in the same way, using in addition Breiman's law of large numbers (see e.g.~ \cite[Corollary 3.3]{BQ3}) and the corresponding empirical measure equidistribution results of Benoist--Quint.
\end{remark}

\begin{proof}
It is clear that any limit point $\nu$ of $\frac{1}{n}\sum_{k=1}^n \mu^{\ast k}\ast \delta_x$ is a $\mu$-stationary probability measure. By Theorem \ref{thm.irreducible.H.decompsable}, there exists a standard trivialization yielding  $H$-equivariant projections on $\pi_1:X \to G/Q$ and $\pi_2:X \to S/\Lambda$, where equivariance in the latter is with respect to a morphism $H \to S$. As a result, a limit point $\nu$ projects via $\pi_1$ and $\pi_2$ to the limit points of $\frac{1}{n}\sum_{k=1}^n \mu^{\ast k}\ast \delta_{\pi_1x}$ and $\frac{1}{n}\sum_{k=1}^n \mu^{\ast k}\ast \delta_{\pi_2x}$, respectively. However, by the uniqueness of Furstenberg measure, the first sequence admits the Furstenberg measure $\overline{\nu}_F$ as a limit. Moreover, by \cite[Theorem 1.5]{prohaska-sert-shi}, the second sequence also admits a limit $\nu^{F}$ which is a homogeneous probability measure on $S/\Lambda$. Since the factor $H/Q_H$ is $\mu$-proximal, it follows by the bijection in Proposition \ref{prop.decomposable.measure.class} that $\nu$ is the unique coupling of $\overline{\nu}_F$ and $\nu^{F}$, i.e.~ the product $\overline{\nu}_F \otimes \nu^{F}$.\end{proof}


\subsection{Equidistribution for diagonal fiber actions (Case 2.2)}\label{subsec.equidist.diagonal}

As mentioned above, unlike the previous cases, the equidistribution problem for the diagonal fiber actions case does not boil down to the corresponding work of Benoist--Quint and we now proceed with our result in this case. 

Recall from Case 2.2 and Lemma \ref{lemma.its.iwasawa} that we have a standard trivialization $X_\calC\simeq \calC\times_\alpha S/\Lambda$ such that the action of $H$ on the fibre $S/\Lambda$ is by a one-dimensional split subgroup $D^\pm$ of $S$ through the Iwasawa cocycle $\alpha$ up to a sign.

The main statement for $\PGL_2(\R)$-case is given in the introduction. Here is the statement for $\SL_2(\R)$ case.
\begin{theorem}\label{thm.equidist.geod.sl}
Keep the hypotheses and notation of Theorem \ref{thm.measure.class.geod} and let $X_\mathcal{C} \simeq H/Q_H \times S/\Lambda$ be the trivialization given by Theorem \ref{thm.measure.class.geod}. Suppose in addition that $H\simeq \SL_2(\R)$ and the measure $\mu$ has finite exponential moment. Suppose $\Gamma_\mu$ preserves a proper closed cone in $\R^2$. Then, the $D$-orbit of $z \in S/\Lambda$ equidistribute to a probability measure $m$ on $S/\Lambda$ if and only if for any $x=(\theta,z)\in X_\calC$ with $\theta$ inside the support of the Furstenberg measure, we have the convergence
\[ \frac{1}{n}\sum_{k=1}^n\mu^{*k}*\delta_x\rightarrow \bar{\nu}_F\otimes m  \quad  \text{as} \; \; n \to \infty. \]
If $\Gamma_\mu$ does not preserve a proper closed cone in $\R^2$, then the $D^\pm$-orbit of $z \in S/\Lambda$ equidistribute to a probability measure $m$ on $S/\Lambda$ if and only if for any $x=(\theta,z)\in X_\calC$, we have the convergence
\[ \frac{1}{n}\sum_{k=1}^n\mu^{*k}*\delta_x\rightarrow \bar{\nu}_F\otimes m  \quad  \text{as} \; \; n \to \infty. \]
\end{theorem}


\subsubsection{Equidistribution result on $K\times \R$}

For $H\simeq \PGL_2(\R)$, if $\mu$ is supported on $\PSL_2(\R)$, then there is no sign issue thanks to the choice of the section $s$ (as taking values in $K^o$). We only need to prove equidistribution result on $K^o\times \R$. The proof is the same as the $\SL_2(\R)$-case. We will comment at the end on the changes needed to handle the $\PSL_2(\R)$-case (i.e.~ Theorem \ref{thm.equidist.geod}).

In order to treat the sign part in the cocycle $\alpha$, we start with equidistribution result on $K\times \R$ instead of $H/Q_H\times \R$. 
Recall that for $H \simeq \SL_2(\R)$ and $\Gamma_\mu$ preserves a closed proper cone (Case 2.2.a), we have two $\mu$-stationary and ergodic measures $\nu_1,\nu_2$ on $\mathbb{S}^1$, both are the lifts of the Furstenberg measure on the projective space $\P(V)$.
In this case, there exist two continuous non-negative functions $p_1$ and $p_2$ (for the characterization of $p_1$ and $p_2$, see \cite[Theorem 2.16]{GLP}) on $\mathbb{S}^1$ such that $p_1+p_2=1$, $p_i|{\Supp \nu_j}=\delta_{i,j}$, where $\delta_{i,j}$ is the Kronecker symbol, and for $j=1,2$, and $x\in \mathbb{S}^1$, we have
\[ p_j(x)=\int p_j(gx)\,d\mu(g). \]
Otherwise (Case 2.2.b), there exists a unique $\mu$-stationary measure $\nu_K$ on $ \mathbb{S}^1$.

Let us define the following measures $\nu_x$:
\begin{definition}\label{defi:nux}
	For $x\in \mathbb{S}^1$, we define
	\[\nu_x :=p_1(x)\nu_1+p_2(x)\nu_2 \, \text{ in Case 2.2.a,  otherwise } \nu_x=\nu_K.\]
\end{definition}
According to \cite[Theorem 2.16]{GLP}, these measures $\nu_x$ are the limit distributions for the random walk on $\bbS$ starting from $x$, following the law of $\mu$.

For the probability measure $\mu$, let $\lambda_\mu$ be its Lyapunov exponent, defined as the almost sure limit of $\frac{1}{n}\log\|g_1\cdots g_n\|$ where $g_1,\cdots,g_n$ are i.i.d.~ random variables with  distribution $\mu$. Let $\sigma_\chi(g,x)=\bar\chi(\bar\sigma(g,\eta))$ for $g\in H$, $x\in \mathbb{S}^1$ and $\eta=\R x\in H/Q_H$, where $\bar\chi(\bar\sigma(g,\eta))$ is defined in Lemma \ref{lemma.iwasawa.norm}.  Clearly, $\sigma_\chi$ does not depend on the lift $x$ of $\eta$ to $\mathbb{S}^1$, so we sometimes use equivalently $\eta$ in the second coordinate to ease the notation. 

\begin{proposition}\label{prop.equidistribution}
Under the same assumptions as in Theorem \ref{thm.equidist.geod.sl}, there exist $\gamma >0$ and $\eta>0$ such that the following holds. For $n\in\N$, $t=\lambda_\mu n$, $\lambda_\mu/2>\eps_1>2/n$, for any $\varphi\in C^3(\bbS\times \R)$ and for $\z \in \bbS$
\begin{equation} \label{eq.renewal+LDP}
\begin{aligned}
\frac{1}{n}\sum_{k=1}^n\int \varphi(g\z ,\sigma_\chi(g,\z ))\ d\mu^{*k}(g)&=\frac{1}{t}\int_{\bbS}\int_0^t \varphi(y,s)\ d s\ d \nu_\z (y)\\
&+O(e^{-\eta\eps_1 n}|\varphi|_{C^3}+|\varphi|_\infty\eps_1+\frac{C|\varphi|_\infty}{n(1-e^{-c})}),
\end{aligned}
\end{equation}
where the constants $C,c>0$ come from the large deviation estimates with rate $\varepsilon_1$ (see Theorem \ref{thm:LDP}).
\end{proposition}

The proof of Proposition \ref{prop.equidistribution} mainly uses the renewal theorem to get the equidistribution and large deviation bounds to get some control of the error.
	
\begin{remark}[Error term]
In order to get a rate in the convergence, we need to know the dependence of the constants $C$ and $c>0$ on $\eps_1$. When $\mu$ has bounded support, both constants can be estimated with $C$ bounded and $c$ quadratic in $\varepsilon_1$, see \cite[Proposition 1.13]{aoun-sert.concentration} which provides subgaussian concentration estimates. In this case, we can get an explicit error term $O(n^{-1/3}|\varphi|_{C^3})$ in Proposition \ref{prop.equidistribution}. With exponential moment, $c$ can still be shown to be quadratic in $\varepsilon_1$ (locally). On the other hand, it might also be possible to use large deviations bounds in a more clever way to get a better error term.
\end{remark}
	
We now proceed to prove Proposition \ref{prop.equidistribution}.
For a function $\varphi$ on $\bbS\times \R$, we define its $L^1C^\gamma$ 
norm by
\[|\varphi|_{L^1C^\gamma}=\int_{\R}\|\varphi(\cdot ,s)\|_{C^\gamma(\bbS)}ds,\  \]
and its $W^{1,2}C^\gamma$ norm by
\[|\varphi|_{W^{1,2}C^\gamma}=|\partial_{ss}\varphi|_{L^1C^\gamma}+|\varphi|_{L^1C^\gamma}, \]
where $C^\gamma$ is the $\gamma$-H\"older norm.
The first ingredient of the proof of Proposition \ref{prop.equidistribution} is the following uniform quantitative renewal theorem which was first proven in \cite{jialun.ens}. We borrow the current version from \cite{jialun.advances}. 
\begin{theorem}\cite[Proposition 5.4]{jialun.advances}\label{thm.renewal}
Under the same assumptions as in Theorem \ref{thm.equidist.geod.sl}, we have the following.
For a compactly supported $C^3$ function $f$ on $\bbS\times\R$, define the renewal sum for $\z \in \bbS$ and $t\in\R^+$ by
\[Rf(\z ,t)=\sum_{k=1}^\infty \int f(g\z ,\sigma_\chi(g,\z )-t)d\mu^{*k}(g). \]
Then, there exists $\eta>0$ such that
\begin{equation}\label{equ_renewal}
Rf(\z ,t)=\frac{1}{\lambda_\mu}\int_{\bbS} \int_{-t}^\infty f(y,u)\ d Leb(u)\ d\nu_\z (y)+O(e^{-\eta (t-|\Supp f|)}|f|_{W^{1,2}C^\gamma}),
\end{equation} 
where 
$$|\Supp f|=\sup\{|s| : (\z ,s)\in\Supp f \text{ for some }\z \in \bbS\}.$$ 
\end{theorem}
A crucial point in this theorem is that the error term is of the form $e^{-\eta(t-|\Supp{f}|)}$, which enables us to take $f$ with support of size $(1-\eps)t$.

The second ingredient of the proof of Proposition \ref{prop.equidistribution} is the following large deviation estimate; we borrow the precise statement from \cite[Thm. 13.11 (iii)]{bq.book}.

\begin{theorem}[Le Page]\label{thm:LDP}
For every $\eps_1>0$, there exist constants $C>0$ and $c>0$ such that
\[\mu^{*n}\{g\in G  :  |\sigma_\chi(g,\z )-\lambda_\mu n|\leq \eps_1 n \}\leq Ce^{-cn}. \]
\end{theorem}

We can now give
	
\begin{proof}[Proof of Proposition \ref{prop.equidistribution}]
We fix $n \in \N$ large enough so that $\lambda_\mu \geq 5/n$ (recall that the Lyapunov exponent $\lambda_\mu$ is  positive, a well-known result of Furstenberg), fix $\varepsilon_1$, $\varphi$ and $\omega$ as in the statement. We will estimate the left-hand side of \eqref{eq.renewal+LDP} separately for $\sigma_\chi(g,\z )$ inside three different intervals $[(\lambda_\mu-\eps_1)n,\infty)$, $[\eps_1n,(\lambda_\mu-\eps_1)n]$ and $(-\infty,\eps_1n]$. The second interval will give us the main term, other intervals will yield the error term. Take a smooth cutoff $\chi$ which equals $1$ on $[\eps_1n,(\lambda_\mu -\eps_1)n]$ and equals $0$ outside of $[\eps_1n-1,(\lambda_\mu -\eps_1)n+1]$ so that we have $\mathds{1}-\chi\leq\mathds{1}_{s<\eps_1n}+\mathds{1}_{s>(\lambda_\mu-\eps_1)n} $. Then, we can write
\begin{equation}\label{eq:split}
\begin{split}
&\left|\frac{1}{n}\sum_{k=1}^n\int \varphi(g\z ,\sigma_\chi(g,\z ))\ d\mu^{*k}(g)-\frac{1}{n}\sum_{k=1}^n\int \varphi(g\z ,\sigma_\chi(g,\z ))\chi(\sigma_\chi(g,\z ))\ d\mu^{*k}(g)\right|\\
\leq & \frac{1}{n}\left|\sum_{k=1}^n\int \varphi(g\z ,\sigma_\chi(g,\z ))\mathds{1}_{\sigma_\chi(g,\z )<\eps_1n}\ d\mu^{*k}(g)\right|\\
&+\frac{1}{n}\left|\sum_{k=1}^n\int \varphi(g\z ,\sigma_\chi(g,\z ))\mathds{1}_{\sigma_\chi(g,\z )>(\lambda_\mu-\eps_1)n}\ d\mu^{*k}(g)\right|.
\end{split}
\end{equation}
	
\textbf{Main term: }
Let $t=n\lambda_\mu$ and $f(\z ,s)=\varphi(\z ,s+t)\chi(s+t)$. Then by \eqref{equ_renewal} \begin{equation}\label{eq:main-renewal}
\begin{split}
&\frac{1}{n}\sum_{k=1}^\infty\int \varphi(g\z ,\sigma_\chi(g,\z ))\chi(\sigma_\chi(g,\z ))\ d\mu^{*k}(g)\\
=&\frac{1}{n}\sum_{k=1}^\infty\int f(g\z ,\sigma_\chi(g,\z )-t)\mu^{*k}(g)\\
=&\frac{1}{t}\int_{\bbS}\int_{-t}^\infty f\ dLeb\ d\nu_\z +\frac{1}{n}O\big(e^{-\eta(t-|\Supp f|)}|f|_{W^{1,2}C^\gamma}\big),
\end{split}
\end{equation}
where in the error term, we have $t-|\Supp f|=t-(t-\eps_1 n)=\eps_1n$. For the main term, using the formula of $f$, we have
\begin{align*}
&\frac{1}{n\lambda_\mu}\int_{\bbS}\int_0^\infty \varphi(y,s)\chi(s)\ dLeb(s)d\nu_\z(y)\\
=&\frac{1}{n\lambda_\mu}\int_{\bbS}\int_{\eps_1n}^{(\lambda_\mu-\eps_1)n} \varphi(y,s)\chi(s)\ dLeb(s)d\nu_\z(y) +|\varphi|_\infty\frac{2}{n\lambda_\mu}\\
=&\frac{1}{t}\int_{\bbS}\int_0^{t}\varphi(y,s)\ dLeb(s)\ d\nu_\z (y)+|\varphi|_\infty O(\eps_1+\frac{1}{n}).
\end{align*}
For the error term in \eqref{eq:main-renewal}, we have
\[ \frac{1}{n}|f|_{W^{1,2}C^\gamma}\leq \sup_s\{|\varphi|_{C^\gamma}\}+\sup_s\{ |\partial_{ss}\varphi|_{C^\gamma}\}\leq |\varphi|_{C^3}.\]
	
Now, we give an upper bound of the sum over  ${k> n}$:
\begin{align*} &\frac{1}{n}\sum_{k>n}^\infty\int \varphi(g\z ,\sigma_\chi(g,\z ))\chi(\sigma_\chi(g,\z ))d\mu^{*k}(g) \\
\leq & |\varphi|_\infty \frac{1}{n}\sum_{k>n}^\infty\mu^{*k}(\{g,\ \sigma_\chi(g,\z )<(\lambda_\mu-\eps_1)n+1\}).
\end{align*}
Due to the assumption $\eps_1 n\geq 2$, we obtain
\[\sigma_\chi(g,\z )-\lambda_\mu n\leq -\eps_1n+1\leq -\eps_1n/2. \]
We use the large deviation estimate (Theorem \ref{thm:LDP}) to obtain
\begin{equation*}
\frac{1}{n}\sum_{k>n}^\infty\mu^{*k}(\{g : \sigma_\chi(g,\z )<(\lambda_\mu-\eps_1)n+1\})\leq \frac{1}{n}\sum_{k>n}Ce^{-ck}=\frac{Ce^{-cn}}{n(1-e^{-c})},
\end{equation*}
where the constants $C,c$ depend on $\eps_1$. 
	
Collecting above estimates, we obtain
\begin{equation}\label{eq:second}
\begin{split}
\frac{1}{n}\sum_{k=1}^n\int \varphi(g\z ,\sigma_\chi(g,\z ))\chi(\sigma_\chi(g,\z ))d\mu^{*k}(g)=&\frac{1}{t}\int_{\bbS}\int_0^t \varphi(y,s)\ dLeb(s)\ d\nu_\z (y)\\
&\!\!\!\!\!\!\!\!\!\!\!\!\!\!\!\!\!\!\!\!\! +O\left(e^{-\eta\eps_1 n}|\varphi|_{C^3}+|\varphi|_\infty  \left(\eps_1+\frac{1}{n}+\frac{Ce^{-cn}}{n(1-e^{-c})}\right)\right).
\end{split}
\end{equation}

\textbf{Error term I: }
For $k<n_0:=\frac{\lambda_\mu-\eps_1}{\lambda_\mu+\eps_1}n$, we have $k(\lambda_\mu+\eps_1)<(\lambda_\mu-\eps_1)n$.
By the large deviation estimates (Theorem \ref{thm:LDP}), we have
\begin{align*}
&\frac{1}{n}\sum_{k=1}^{n_0}\int \varphi(g\z ,\sigma_\chi(g,\z ))\mathds{1}_{\sigma_\chi(g,\z )\geq (\lambda_\mu-\eps_1)n}d\mu^{*k}(g)\\
\leq&  |\varphi|_\infty \frac{1}{n}\sum_{k=1}^{n_0}\mu^{*k}(\{g\in G : \sigma_\chi(g,\z )>(\lambda_\mu+\eps_1)k \})\\
\leq & |\varphi|_\infty C\frac{1}{n}\sum_{k=1}^{n_0} e^{-ck}\leq \frac{|\varphi|_\infty C}{n(1-e^{-c})}.
\end{align*}

For the part $n_0\leq k\leq n$, we use the absolute value to bound
\[\frac{1}{n}\sum_{k=n_0}^{n}\int \varphi(g\z ,\sigma_\chi(g,\z ))\mathds{1}_{\sigma_\chi(g,\z )\geq (\lambda_\mu-\eps_1)n}d\mu^{*k}(g)\leq |\varphi|_\infty\frac{n-n_0}{n}=|\varphi|_\infty\frac{2\eps_1}{\lambda_\mu+\eps_1}. \]
Thus, we have
\begin{equation}\label{eq:first}
\frac{1}{n}\sum_{k=1}^{n}\int \varphi(g\z ,\sigma_\chi(g,\z ))\mathds{1}_{\sigma_\chi(g,\z )\geq (\lambda_\mu-\eps_1)n}d\mu^{*k}(g)\leq  \frac{|\varphi|_\infty C}{n(1-e^{-c})}+|\varphi|_\infty\frac{2\eps_1}{\lambda_\mu+\eps_1}.
\end{equation}
	
	
\textbf{Error term II: }If $k>n_1:=\eps_1n/(\lambda_\mu-\eps_1)$, then we have $\eps_1n<k(\lambda_\mu-\eps_1)$ and hence we can apply the large deviation estimate to obtain
\begin{align*}
&\frac{1}{n}\sum_{k=n_1}^n\int \varphi(g\z ,\sigma_\chi(g,\z ))\mathds{1}_{\sigma_\chi(g,\z )\leq \eps_1 n}d\mu^{*k}(g)\\
\leq & |\varphi|_\infty \frac{1}{n}\sum_{k=n_1}^n\mu^{*k}(\{g\in G : \sigma_\chi(g,\z )<(\lambda_\mu-\eps_1)k \})\leq \frac{C|\varphi|_\infty e^{-cn_1}}{n(1-e^{-c})}.
\end{align*}
For the part $k\leq n_1$, 
\begin{align*}
\frac{1}{n}\sum_{k=1}^{n_1}\int \varphi(g\z ,\sigma_\chi(g,\z ))\mathds{1}_{\sigma_\chi(g,\z )\leq \eps_1 n}d\mu^{*k}(g)\leq|\varphi|_\infty\frac{n_1}{n}=|\varphi|_\infty\frac{\eps_1}{\lambda_\mu-\eps_1}.
\end{align*}
Thus, we have
\begin{equation}\label{eq:thrid}
\frac{1}{n}\sum_{k=1}^{n}\int \varphi(g\z ,\sigma_\chi(g,\z ))\mathds{1}_{\sigma_\chi(g,\z )\leq \eps_1 n}d\mu^{*k}(g)\leq|\varphi|_\infty\frac{\eps_1}{\lambda_\mu-\eps_1}+\frac{C|\varphi|_\infty e^{-cn_1}}{n(1-e^{-c})}.
\end{equation}
Finally, combining \eqref{eq:split}, \eqref{eq:second}, \eqref{eq:first} and \eqref{eq:thrid}, we obtain
\begin{align*}
\frac{1}{n}\sum_{k=1}^\infty\int \varphi(g\z ,\sigma_\chi(g,\z ))\ d\mu^{*k}(g)=\frac{1}{t}\int_0^t \varphi(y,s)\ dLeb(s)\ d\nu_\z (y)\\
+O\left(e^{-\eta\eps_1 n}|\varphi|_{C^3}+|\varphi|_\infty\frac{\eps_1}{\lambda_\mu-\eps_1}+\frac{C|\varphi|_\infty}{n(1-e^{-c})}\right).
\end{align*}
\end{proof}

\subsubsection{Equidistribution on $X_\calC$}
We now use the equidistribution on $K\times \R$ (Proposition \ref{prop.equidistribution}) to deduce the equidistribution on $X_\calC$, that is, to give

\begin{proof}[Proof of Theorem \ref{thm.equidist.geod.sl}]
Let $K\times_\sigma D$ be the fiber bundle with $H$ action, the action of $H$ is given by $h(k,d)=(hk,\sigma(h,k)d)$, where we identify $K \simeq \mathbb{S}^1 \simeq H/AN$. We define a map $p$ from $K\times_\sigma D$ to $H/Q_H\times_\alpha D^\pm$ by 
\[ p(k,d)=(kM,\sg(k)d), \]
where $\sg(k)$ is the sign element in $M$. By Lemma \ref{lem:sign}, we have
\begin{lemma}\label{lem.hequivariant}
The map $p$ is an $H$-equivariant map from $K\times_\sigma D$ to $H/Q_H\times_\alpha D^\pm$.
\end{lemma}

We denote by $\calG(r,\sg(w))$ the element $(e^r,\sg(w))\in D^\pm$ and we use additive parameter $r\in \R$. Under this parametrization, the equidistribution of $D$ or $D^\pm$-orbits of $z \in S/\Lambda$ to some measure $m$ means, respectively, that the measure $\frac{1}{t}\int_0^t\delta_{\calG(r,1)z}\ dr$ or the measure $\frac{1}{2t}\int_0^t\delta_{\calG(r,1)z}+\delta_{\calG(r,-1)z}\ dr$ converges to $m$ as $t \to \infty$.

Let $\psi$ be $C^3$ function on $X_\mathcal{C} \simeq H/Q_H \times S /\Lambda$ and $z \in S/\Lambda$. Set $\varphi(w,r):=\psi(\eta,\calG(r,\sg(w))z)$ where $\eta$ is the projection of $w$ on $H/Q_H \simeq K/M$. Thanks to Lemma \ref{lem.hequivariant}, we have the relation 
\begin{equation}
\varphi(g s(\eta),\sigma_\chi(g,\eta))=\psi(g\eta,\alpha(g,\eta)z)=\psi(gx),
\end{equation}
for $x=(\eta,z) \in X_\mathcal{C}$ and where, we recall, $s:K/M \to K$ is the section. Therefore, we have
\begin{equation}\label{eq.corresp.phi.psi}
\frac{1}{n}\sum_{1\leq k\leq n}\mu^{*k}*\delta_x(\psi)=\frac{1}{n}\sum_{k=1}^n\int \varphi(g s(\eta) ,\sigma_\chi(g,\eta))\ d\mu^{*k}(g).
\end{equation}

We know that for any $\psi\in C^3(X_\calC)$ with bounded $C^3$ norm, with suitable choice of $\eps_1$ depending on $n$, thanks to Proposition \ref{prop.equidistribution} and the relation \eqref{eq.corresp.phi.psi}, for $t=\lambda_\mu n$, we have
\[ \frac{1}{n}\sum_{1\leq k\leq n}\mu^{*k}*\delta_x(\psi)- \frac{1}{t}\int_{\bbS}\int_0^t\varphi(k',r)\ dr\ d\nu_{s(\eta)}(k')\rightarrow 0 \]
as $n$, equivalently $t$, tends to $\infty$.

On the other hand, by construction of the function $\varphi(\cdot,\cdot)$ and the measures $\nu_w$ for $w \in \mathbb{S}^1$, we have
\begin{align*}
\int_{\mathbb{S}^1} \varphi(k',r)\ d\nu_w(k')=& p_1(w)\int_{\bbS} \psi(k'M,\calG(r,\sg(k'))z)\ d\nu_1(k')\\
&+p_2(w)\int_{\mathbb{S}^1} \psi(k'M,\calG(r,\sg(k'))z) d\nu_2(k')\\
=&p_1(w)\int \psi(\eta,\calG(r,1)z)\ d\overline{\nu}_F(\eta)+p_2(w)\int \psi(\eta,\calG(r,-1)z)\ d\overline{\nu}_F(\eta).
\end{align*}
We get that
\begin{equation*}
    \frac{1}{n}\sum_{1\leq k\leq n}\mu^{*k}*\delta_x(\psi)- \frac{1}{t}\int \int_0^t p_1(s(\eta))\psi(\eta,\calG(r,1)z)+p_2(s(\eta))\psi(\eta,\calG(r,-1)z)\ dr\ d \overline{\nu}_F(\eta)\rightarrow 0.
 \end{equation*}
 
Recall that for Case 2.2.a, the section $s$ is chosen so that its image contains the support of $\nu_1$. In particular, $p_1(s(\eta))=1$ and $p_2(s(\eta))=0$ for every $\eta$ in the support of the Furstenberg measure $\overline{\nu}_F$. Therefore, in Case 2.2.a, if the $\eta$ coordinate of $x=(\eta,z)$ belongs to the support of $\overline{\nu}_F$, then we have
\begin{equation}\label{eq.conv1}
    \frac{1}{n}\sum_{1\leq k\leq n}\mu^{*k}*\delta_x(\psi)- \frac{1}{t}\int \int_0^t \psi(\eta,\calG(r,1)z)\ dr\ d\bar\nu_F(\eta)\rightarrow 0.
 \end{equation}
In Case 2.2.b, by a similar computation and using the fact that the unique measure $\nu_K$ on $\mathbb{S}^1$ writes as $\nu_K=\frac{1}{2} \int  (\delta_w + \delta_{-w}) d\overline{\nu}_F(\R w)$, we have
\begin{equation}\label{eq.conv2}
    \frac{1}{n}\sum_{1\leq k\leq n}\mu^{*k}*\delta_x(\psi)- \frac{1}{2t}\int \int_0^t (\psi(\eta,\calG(r,1)z)+\psi(\eta,\calG(r,-1)z)) \ dr\ d\overline{\nu}_F(\eta)\rightarrow 0.
 \end{equation}
By density of $C^3(X_\calC)$ in $C(X_\calC)$, we deduce from \eqref{eq.conv1} and \eqref{eq.conv2} that $ \frac{1}{n}\sum_{1\leq k\leq n}\mu^{*k}*\delta_x$ converges weakly to a measure $\bar\nu_F\otimes m$ if and only if the $D$-orbit or the $D^{\pm}$- orbit (respectively in Case 2.2.a or Case 2.2.b) starting at $z \in S/\Lambda$ equidistributes to the measure $m$.
\end{proof}
 
\begin{proof}[Proof of Theorem \ref{thm.equidist.geod}] 
For $\Gamma_\mu< \PSL_2(\R)\simeq \PGL_2(\R)^o$, we take the lift $\mu_1$ on $\SL_2(\R)$ of $\mu$ on $\PSL_2(\R)$ with equal probability on two preimages of each element. Then we can apply Theorem \ref{thm.renewal} to this new measure $\mu_1$. Here the element $\diag(-1,-1)$ (which is the non-trivial element of $M$ in the case of $H \simeq \SL_2(\R)$) maps to identity in $G$, which acts trivially. Then the same argument as in the proof of Theorem \ref{thm.renewal} readily yields Theorem \ref{thm.equidist.geod}.
\end{proof}



\end{document}